\documentclass[11pt]{article}
\usepackage[utf8]{inputenc}

\title{Lower Central Series Ideal Quotients Over $\mathbb{F}_p$ and $\mathbb{Z}$}

\usepackage{epigraph}
\usepackage{color}
\usepackage[square,sort,comma,numbers]{natbib}
\usepackage{graphicx}
\usepackage{amsmath}
\usepackage{amsfonts}
\usepackage{amssymb}
\usepackage{floatrow}
\usepackage{fullpage,setspace}
\usepackage{hyperref} 
\usepackage{authblk}

\newtheorem{theorem}{Theorem}[section]
\newtheorem{fact}{Fact}[section]
\newtheorem{lemma}[theorem]{Lemma}
\newtheorem{proposition}[theorem]{Proposition}
\newtheorem{corollary}[theorem]{Corollary}
\newtheorem{definition}[theorem]{Definition}

\newenvironment{proof}[1][Proof]{\begin{trivlist}
\item[\hskip \labelsep {\bfseries #1}]}{\end{trivlist}}

\newenvironment{remark}[1][Remark]{\begin{trivlist}
\item[\hskip \labelsep {\bfseries #1}]}{\end{trivlist}}

\newcommand{\qed}{\nobreak \ifvmode \relax \else
      \ifdim\lastskip<1.5em \hskip-\lastskip
      \hskip1.5em plus0em minus0.5em \fi \nobreak
      \vrule height0.75em width0.5em depth0.25em\fi}

\usepackage[compact]{titlesec}
\titlespacing{\section}{0pt}{*0}{*0}
\titlespacing{\subsection}{0pt}{*0}{*0}
\titlespacing{\subsubsection}{0pt}{*0}{*0}

\author[2,3]{Yael Fregier\thanks{yael.fregier@gmail.com}}
\author[]{Isaac Xia\thanks{isaacxia@college.harvard.edu}}
\affil[2]{Max Planck Institute for Mathematics, Bonn}
\affil[3]{LML, Artois University}

\date{}
\begin{document}
\pagenumbering{gobble}
\maketitle
\onehalfspacing

\begin{abstract}

Given a graded associative algebra $A$,  its lower central series is defined by $L_1 = A$ and $L_{i+1} = [L_i, A]$.  We consider successive quotients $N_i(A) = M_i(A) / M_{i+1}(A)$, where $M_i(A) = AL_i(A) A$.  These quotients are direct sums of graded components.  Our purpose is to describe the $\mathbb{Z}$-module structure of the components; i.e., their free and torsion parts.  Following computer exploration using {\it MAGMA}, two main cases are studied.  The first considers $A  =  A_n / (f_1,\dots, f_m)$, with $A_n$ the free algebra on $n$ generators $\{x_1, \ldots, x_n\}$ over a field of characteristic $p$. The relations $f_i$ are noncommutative polynomials in $x_j^{p^{n_j}},$ for some integers $n_j$.  For primes $p > 2$, we prove that $p^{\sum n_j} \mid \text{dim}(N_i(A))$.  Moreover, we determine polynomials dividing the Hilbert series of each $N_i(A)$.  The second concerns $A = \mathbb{Z} \langle x_1, x_2, \rangle / (x_1^m, x_2^n)$.  For $i = 2,3$, the bigraded structure of $N_i(A_2)$ is completely described.  
\end{abstract}

\section{Introduction}

Algebraic geometry is technically based on commutative algebra as one can reconstruct an affine algebraic variety from its commutative algebra of functions. This suggests to define a noncommutative ``space" via a noncommutative algebra which plays the role of the algebra of functions on this nonexistent space.

This can seem a very daring postulate, but it has proven to be a powerful one. It lies at the heart of the theory of  noncommutative geometry of Alain Connes and  Quantum groups of Vladimir Drinfeld.\\

Feigin and Shoikhet \cite{FS} initiated a new approach to the study of a given noncommutative algebra. Their idea was to approximate it by pieces whose degree of noncommutativity is controlled. This parallels the idea of approaching a function by polynomials in its Taylor expansion. One gains through these ``more commutative" approximations an access to tools of classical geometry.  

To be more precise, the first approximation of a noncommutative algebra $A$ is its abelianization $A_{ab}:=A/A[A,A]A.$ To generalize this construction to higher orders one can consider the {\it lower central series} $(L_i)_{i\in\mathbb{N}}$. It is defined inductively: the first term $L_1$ is  $A$ itself, while the following ones are defined as $L_{i+1}=[A,L_i]$. In particular the abelianization of $A$ can be interpreted as
$A_{ab}=A/A[A,A]A=M_1/M_2,$ where $M_i$ denotes the ideal generated by $L_i$, i.e. $M_i:=AL_iA$. This suggests to define
$N_i:=M_i/M_{i+1}$ as a generalization of $A_{ab}$.  Note that some other papers on the same subject define and study directly $B_i:=L_i/L_{i+1}$, without first forming an ideal.

The innovative work of Feigin and Shoikhet spawned a new line of research.  The structure of $B_i(A)$ was first studied by \cite{FS}, then by \cite{DKME}, \cite{DE}, \cite{AJ}, \cite{BaJ}, \cite{BB}, \cite{BoJ}, and \cite{BEJ+}.  Shortly after came the study of the $N_i(A)$, including papers by \cite{Ker}, \cite{BEJ+}, \cite{JO}, and lastly \cite{CFZ}. 
 
In their paper, \cite{FS} considered $A=A_n(\mathbb{C})$, the free associative algebra on $n$ letters, over the field of complex numbers, but their results remain valid over any field of characteristic zero, in particular over $\mathbb{Q}$. They have discovered that $A/M_3$ can be identified with the algebra $\Omega_{even}(\mathbb{C}^n)$ of even differential forms on $\mathbb{C}^n$ with Fedosov product.  Thus, one can wonder whether there are other incarnations of classical geometric objects hidden in the $N_i(A)$'s.

 This is a difficult question, and a first approach to understand the $N_i(A)$'s is to determine their dimensions.  We do not want to restrict ourselves to free algebras, but consider instead algebras with relations. We work with fields or rings different than $\mathbb{Q}$, for example over the integers $\mathbb{Z}$ or a finite field $k$ of characteristic $p$, as these are more accessible to computer assisted exploration.\\

  In the first section, we consider algebras of the form $A := A_n / (f_1, f_2, \ldots, f_m)$.  We show in Theorem \ref{WNkNiAaction} that $W_n(k)$, the Weyl algebra with divided powers, acts on $N_i(A)$. More generally there is an action of  $W_{n_1}(k)\otimes,..\otimes W_{n_r}(k)$, and one obtains (corollary \ref{NiDivisPN}) that $\text{dim}(N_i(A))$ is divisible by $p^{\sum n_j}$. We also deduce (corollary \ref{HilbertSeriesDivisN}) that the Hilbert series of $N_i(A)$ with respect to the corresponding variables $X_1, \ldots, X_r$ is divisible by 
$(1+X_1 + \cdots + X_1^{p^{n_1} - 1})\cdots(1+X_r + \cdots + X_r^{p^{n_r} - 1}).$

In the second section, we work over $\mathbb{Z}$ and consider algebras of the form $A:=A_2 / (x_1^m,x_2^n)$. We prove that the $\mathbb{Z}$-module structure of $N_2(A)$ and $N_3(A)$ are given by the tables (see notations in \ref{N2} and \ref{gcdDef})
\begin{table} [H] 
\small
\caption{Bigraded Description of $N_2 (A)$}
\label{Bigraded Description of $N_2 (A)$}
\begin{tabular}{|l||l|l|l|ll|l|l|}
\hline
$(m,n)$ & $0$ & $1$ & $2$ & $\cdots$ & $\cdots$ & $n-1$ & $n$ \\
\hline
\hline
$0$ & $0$ & $\cdots$ & $\cdots$ & $\cdots$ & $\cdots$ & $\cdots$ & $\cdots$  \\
\hline
$1$ & $\vdots$ & $\mathbb{Z}$ & $\mathbb{Z}$ & $\cdots$ & $\cdots$ & $\mathbb{Z}$ & $\mathbb{Z}_n$ \\
\hline		s
$2$ & $\vdots$ & $\mathbb{Z}$ & $\mathbb{Z}$ &$\cdots$ & $\cdots$ & $\mathbb{Z}$ & $\mathbb{Z}_n$  \\
\hline
$\ \vdots$ & $\vdots$ & $\vdots$ & $\vdots$ & $\ddots$ &  & $\vdots$ & $\vdots$  \\

$\ \vdots$ & $\vdots$ & $\vdots$ & $\vdots$ & & $\ddots$ & $\vdots$ & $\vdots$   \\
\hline
$m-1$ & $\vdots$ & $\mathbb{Z}$ & $\mathbb{Z}$ & $\cdots$ & $\cdots$ & $\mathbb{Z}$ & $\mathbb{Z}_n$\\ \hline
$m$ & $\vdots$ & $\mathbb{Z}_m$ & $\mathbb{Z}_m$ & $\cdots$ & $\cdots$ & $\mathbb{Z}_m$ & $\mathbb{Z}_{(m,n)}$  \\
\hline

\end{tabular}
\end{table}

and

\begin{table} [H] 
\small
\caption{Bigraded Description of $N_3 (A)$}
\label{Bigraded Description of $N_3 (A)$}
\begin{tabular}{|l||l|l|l|ll|l|l|l|}
\hline
$(m,n)$ & $0$ & $1$ & $2$ & $\cdots$ &  $\cdots$ & $n-1$ & $n$ & $n+1$  \\
\hline
\hline
$0$ & $0$ & $\cdots$ & $\cdots$ & $\cdots$ & $\cdots$ & $\cdots$ & $\cdots$ & $\cdots$ \\
\hline
$1$ & $\vdots$ & $0$ & $\mathbb{Z}$ & $\cdots$ & $\cdots$ & $\mathbb{Z}$ & $\mathbb{Z}$ & $\mathbb{Z}_{f(n)}$\\
\hline
$2$ & $\vdots$ & $\mathbb{Z}$ & $\mathbb{Z}^3$ & $\cdots$ & $\cdots$ & $\mathbb{Z}^3$ & $\mathbb{Z}^2 \oplus \mathbb{Z}_n$ & $\mathbb{Z}_n \oplus \mathbb{Z}_{f(n)}$  \\
\hline
$\ \vdots$ & $\vdots$ & $\vdots$ & $\vdots$ & $\ddots$ & & $\vdots$ & $\vdots$ & $\vdots$  \\

$\ \vdots$ & $\vdots$ & $\vdots$ & $\vdots$ & & $\ddots$ & $\vdots$ & $\vdots$ & $\vdots$ \\
\hline
$m-1$ & $\vdots$ & $\mathbb{Z} $ & $\mathbb{Z}^3$ & $\cdots$ & $\cdots$ & $\mathbb{Z}^3$ & $\mathbb{Z}^2 \oplus \mathbb{Z}_n$ & $\mathbb{Z}_n \oplus \mathbb{Z}_{f(n)}$ \\ \hline
$m$ & $\vdots$ & $\mathbb{Z}$ & $\mathbb{Z}^2 \oplus \mathbb{Z}_m$ & $\cdots$ & $\cdots$ & $\mathbb{Z}^2 \oplus \mathbb{Z}_m$ & $\mathbb{Z}_m \oplus \mathbb{Z}_n$ & $\mathbb{Z}_{f(n)} \oplus \mathbb{Z}_{(m,n)}$\\
\hline
$m+1$ & $\vdots$ & $\mathbb{Z}_{f(m)}$ & $\mathbb{Z}_m \oplus \mathbb{Z}_{f(m)}$ & $\cdots$ & $\cdots$ & $\mathbb{Z}_m \oplus \mathbb{Z}_{f(m)}$ & $\mathbb{Z}_{f(m)} \oplus \mathbb{Z}_{(m,n)}$ & $\mathbb{Z}_{(m,n)}$ \\
\hline
\end{tabular}
\end{table}

 We give an explicit basis of the non-torsion part and also compute the torsion in terms of $m$ and $n$.\\

\section{Divisibility of Total Dimensions in characteristic $p$}
\label{F_p section}

The main tool of this section, Proposition
\ref{WNkDivispN}, states that any finite dimensional module over $\otimes_j W_{n_j}(k)$, a sub algebra of the Weyl algebra with divided power structure, has dimension divisible by all $p^{n_j}.$ We show in Theorem \ref{WNkNiAaction}  that $N_i(A)$ can be equipped with an action of $\otimes_j W_{n_j}(k)$, and as a corollary, one obtains (corollary \ref{NiDivisPN}) that $\text{dim}(N_i(A))$ is divisible by all $p^{n_j}$.

\subsection{Weyl algebra with divided powers}\label{sectWeyl}

We first recall the definition \ref{Weyl algebra with divided powers} of the algebra  $W(k)$ and then give in lemma \ref{DpiGeneratesW*} a system of generators in order to formulate the definition \ref{Wn} of  $W_n(k)$. 

\begin{definition}
\label{Weyl algebra with divided powers}
The {\bf Weyl algebra with divided powers} over $\mathbb{Z}$, $W (\mathbb{Z})$, is the algebra of linear operators of the form$$\sum_{i,j} a_{ij} x^i \frac{D^j}{j!},$$
where $D := \dfrac{\partial}{\partial x}$ and the coefficients $a_{ij}$ are in $\mathbb{Z}.$ For a commutative ring $R$, one defines $W (R) := W (\mathbb{Z}) \otimes R$.
\end{definition}

  Note that the elements of $W(\mathbb{Z})$ define endomorphisms of $\mathbb{Z}[x]$ despite the denominators. If we denote  $D_j:= D^j/j!$ it is clear that $x$, together with $D_j$ for all non-negative $j$ generate $W (\mathbb{Z})$.  Also one has:
\begin{equation}
\label{Div eq1}
D_j D_r = \dfrac{D^j D^r}{j! r!} = \frac{(j+r)!}{j! r!} \frac{D^{j+r}}{(j+r)!} = \binom{j+r}{j} D_{j+r}.
\end{equation}

 From now on, $R$ will be a field $k$ of characteristic $p$.
We have a well-known Lemma:

\begin{lemma}
\label{DpiGeneratesW*}
If $k$ has characteristic $p$, the algebra generated by $D_j$ for all non-negative $j$ is also generated by $D_{p^i}$ for all non-negative $i$.  More precisely, if $a$ is a non-negative integer with representation $a = a_n p^n + \cdots + a_0$ in base $p$, we have

\begin{equation}
\label{Weyl eq 2}
D_a =  \dfrac{1}{C}\prod_s (D_{p^s})^{a_s}, \text{ with } C = \prod_s (a_s!).
\end{equation}
\end{lemma}

\begin{proof}
One can write $a$ as the sum of $2$ elements $b$ and $c$, in a compatible way with its decomposition in basis $p$:
$$a=\underset{b}{\underbrace{a_np^n+\dots+a_kp^k}}+\underset{c}{\underbrace{a_{k-1}p^{k-1}+\dots+a_0}}.$$
We claim that $$D_a=D_bD_c.$$
According to eq. (\ref{Div eq1}), we already know that
$${a \choose b}D_a=D_bD_c.$$
Therefore it suffices to prove, that ${a \choose b}=1 \pmod{p}.$

Let us recall Lucas's Theorem: for all non-negative integers $m,n$ and prime $p$, we have
\begin{equation}
\label{lucas}
\binom{m}{n} \equiv \prod_{i=0}^k \binom{m_i}{n_i} \pmod{p},
\end{equation}
where $m = \sum_{i = 0}^k m_i p^i$ and $n = \sum_{i=0}^k n_i p^i$.
In our setting:
$${a \choose b}=\Pi_s  {a_s \choose b_s} \pmod{p}.$$
But we can decompose this product into two products (for $s\leq k-1$ and for $s>k-1$) and use the remark that by definition of $b$,
$b_s=\begin{cases}
a_s & \text{if $s>k-1$},\\
0  & \text{otherwise.}
\end{cases}$

In other words (mod $p$):
\begin{eqnarray*}
{a \choose b} & = & \Pi_s  {a_s \choose b_s}\\
                         & = & \Pi_{s> k-1}  {a_s \choose b_s}\Pi_{s\leq k-1}  {a_s \choose b_s}\\
                         & = & \Pi_{s> k-1}  {a_s \choose a_s}\Pi_{s\leq k-1}  {a_s \choose 0}\\
                         & = & 1.
\end{eqnarray*}
Iterating this result, one gets \begin{equation}\label{prod} D_a= \Pi_s D_{a_s p^s}.\end{equation}

We now want to prove by induction that \begin{equation}\label{pow}\alpha !D_{\alpha p^i}= (D_{p^i})^\alpha.\end{equation}
By eq. (\ref{Div eq1}),
$${\alpha p^i \choose p^i}D_{\alpha p^i }=D_{(\alpha-1) p^i }D_{p^i },$$
so we are looking for the expression of ${\alpha p^i \choose p^i}.$
But Lukas' theorem gives $${\alpha p^i \choose p^i}={\alpha  \choose 1}=\alpha,$$ which completes the induction step.  $\square$
\end{proof}

Thus, $W(k)$ is generated by $x$ and $D_{p^i}$ for all $i \ge 0$. 

\begin{definition}\label{Wn}
 Denote by $W_n (k)$ the subalgebra generated by $x$ and $D_{p}, \cdots, D_{p^{n-1}}$.  By Lemma \ref{DpiGeneratesW*}, it is generated by $x$ and all $D_j$ with $j < p^n$.  
 \end{definition}
 
 For example, $W_1(k)$ is generated by $x$ and $D$ with relations $[D,x] = 1$ and $D^p = 0$.

We will need the following lemma in the proof of proposition \ref{WNkDivispN}.

\begin{lemma}
\label{WNkDivispN algebra}
For $j <p^n$, all $D_j \in W_n (k)$ are nilpotent. Moreover $x^{p^n}$ is central in this algebra.
\end{lemma}

\begin{proof}
To show that all $D_j$ are nilpotent, we first show that all $D_{p^i}$ are nilpotent.

Since $D_{(m-1)p^i} D_{p^i} \stackrel{(\ref{Div eq1})}{=} \binom{mp^i}{p^i} D_{m p^i}$, an induction with Lucas's theorem $\binom{mp^i}{p^i} \stackrel{(\ref{lucas})}{=} \binom{m}{1} = m$ shows that $m!D_{mp^i} = D_{p^i}^m$.
In particular, for $m = p$, we have that
\begin{equation}
\label{Nilpotent D_p^i}
D_{p^i}^p = p!D_{p^{i+1}} = 0.
\end{equation}

For arbitrary $0 \le j < p^n$, we have by the proof of Lemma \ref{DpiGeneratesW*} that
$$D_j^p \stackrel{(\ref{prod})}{=}  \left( \prod_s D_{j_s p^s} \right)^p= \prod_s (D_{j_s p^s})^p.$$
  It remains to show that one of the terms in this product vanishes. Choosing any term in the product and noting that since all $j_s < p$, $j_s! \ne 0$, one has
\begin{equation}
\label{Nilpotent D_j}
D_{j_s p^s}^p \stackrel{(\ref{pow})}{=} \left( \dfrac{D_{p^s}^{j_s}}{j_s!} \right)^p  = \dfrac{(D_{p^s}^{j_s})^p}{j_s!^p} = \dfrac{(D_{p^s}^p)^{j_s}}{j_s!^p} \stackrel{(\ref{Nilpotent D_p^i})}{=}  \dfrac{0^{j_s}}{j_s!^p} = 0.
\end{equation}
Thus, we have shown that all $D_j$ are nilpotent.

It is clear that $x^{p^n}$ commutes with $x$. We now show that it commutes with $D_j$ as well. According to lemma \ref{DpiGeneratesW*} it suffices to show it for $D_{p^i}$, with $p^i < p^n$.  To this end, note that
$$[D_{p^i}, x^{p^n}] x^{\ell} = D_{p^i}(x^{p^n} x^{\ell}) - x^{p^n} (D_{p^i} x^{\ell})  = \binom{p^n + \ell}{p^i} x^{p^n + \ell - p^i} - \binom{\ell}{p^i} x^{p^n + \ell - p^i}.$$
Now, we show that $0 = (\binom{p^n + \ell}{p^i} - \binom{\ell}{p^i})$.  But, by Lucas's Theorem we have that
$$\binom{p^n + \ell}{p^i} - \binom{\ell}{p^i} = \binom{1}{0} \prod_{s}^{N-1} \binom{\ell_{s}}{p^{i}_s} - \prod_s^{N-1} \binom{\ell_s}{p^{i}_s} = 0.\  \square$$

\end{proof}

\subsection{Divisibility of dimensions of $\otimes_j W_{n_j}(k)$-modules}\label{tool}

We first recall that the tensor product $A\otimes B$ of two associative algebras $A$ and $B$ is also an associative algebra for the product $(a\otimes b)\cdot (a'\otimes b'):=a\cdot a'\otimes b\cdot b'$. One has a canonical injection of $A$ into $A\otimes B$ sending $a$ to $a\otimes 1_B$ (similarly for $B$). 
\begin{remark}One immediately see that elements of $A$ commute with those of $B$ in $A\otimes B$.
\end{remark}
We can therefore consider the associative algebra $\otimes_j W_{n_j}(k)$ which is generated by the elements
$$D_{i j}:= \underset{i-1}{\underbrace{1\otimes  \dots 1}}\otimes D_j\otimes 1\dots \otimes 1$$ and
$$x_{i }:= \underset{i-1}{\underbrace{1\otimes  \dots 1}}\otimes x \otimes 1\dots \otimes 1.$$

We start by stating a useful lemma

\begin{lemma}\label{commutenil}
Let $N_i$ be a finite family of commuting nilpotent endomorphisms of a vector space $V$, then there exists a non zero vector $v \in V$ annihilated by all the $N_i$'s.
\end{lemma}

\begin{proof}
Our base case is true: as $N_1$ is nilpotent, for any $v$, there exists a certain power $s$ for which $v_1:=N_1^sv_k$ is nonzero, but $N_1^{s+1}v_k$ vanishes, so one has $N_1(v_1)=0.$
Now, suppose that $N_1, \ldots, N_k$ all share a common null vector $v_k \in V$.
Since $N_{k+1}$ is nilpotent, there exists some integer $\ell$ such that $N_{k+1}^{\ell} (v_k) = 0$ and $v_{k+1}:= N_{k+1}^{\ell - 1} (v_k)\neq 0$. In particular  $N_{k+1} (v_{k+1}) = 0$.  Additionally, for any $j \le k$, we have $N_j (v_{k+1}) = N_j (N_{k+1}^{\ell - 1} v_k) = N_{k+1}^{\ell - 1} N_j (v_k) = 0,$ so our induction is done.
\end{proof}

For the rest of this section, we assume that  $k$ is algebraically closed.

\begin{lemma}
\label{WNkDivispN module}
Let $V$ be a $\otimes_j W_{n_j}(k)$ module.  Then all of the $D_{ij}$ share a common null vector $v \in V$. Moreover, if  $V$ is irreducible, each of $x_j^{p^{n_j}}$ act by corresponding scalars $s_j \in k$.
\end{lemma}

\begin{proof}
The $D_{ij} $'s commute with each other by definition for different $j's$ and by the above remark for different $i's$. They are nilpotent by Lemma \ref{WNkDivispN algebra}.  Lemma  \ref{commutenil} therefore applies.

In addition, since each $x_j^{p^{n_j}}$ is central in $\otimes_j W_{n_j}(k)$, and since $V$ is an irreducible $\otimes_j W_{n_j}(k)$-module, Schur's lemma asserts that $x_j^{p^{n_j}}$ acts by multiplication by a scalar. $\square$
\end{proof}

This lemma enables to derive the main result of this section:
\begin{proposition}
\label{WNkDivispN}

Any finite dimensional module over $W_N:=\bigotimes_j W_n(k)$ has dimension divisible by  $p^{\sum n_j}.$
\end{proposition}

We recall the following basic result whose proof we omit:

\begin{lemma}
\label{Mapping between Spaces and Quotients}
Let $E,\ F$ be subspaces of vector spaces $V,\ W$ respectively.  Given a linear mapping $\phi : V \to W$ such that $\phi(E)\subset F$, the map $\bar{\phi} : V/E \to W/F$ given by $\bar{\phi}([v]) = [\phi(v)]$ for $v \in V$ is well defined and linear.  
\end{lemma}

\begin{proof} (of prop. \ref{WNkDivispN}):
Let $M$ be a finite dimensional module over $W_N$.   If $M$ is not already irreducible, then we may find an irreducible submodule $V_1$ of $M$; then, we have that $M \cong V_1 \oplus M/V_1$.  If $M/V_1$ is not yet irreducible, then we may find an irreducible submodule $V_2 \subset M/V_1$; this implies the existence of a module $F_2 \subset M$ such that $F_1:=V_1 \subset F_2$ and $F_2 / F_1 \cong V_2$.  By continuing this process we build in a finite number of steps an exhausting filtration $F_1\subset \dots \subset F_n=M$ of $M$. The associated successive quotients $V_i:=F_{i} / F_{i-1} $ are by construction irreducible modules and together form the Jordan-H\"{o}lder decomposition of $M$ :
$$M = V_1 \oplus V_2 \oplus \cdots \oplus V_{d}.\footnote{$V_2$ may not be a submodule of $M$, as this is a decomposition as vector spaces.}  $$
 To prove the proposition, it suffices to show that each $V_i$ has dimension divisible by $p^{\sum n_j}.$
Let $V$ be one of these $V_i$.

Our strategy is to show that $$V \cong k[x_1, \ldots, x_m]/(x_1^{p^{n_1}}, \ldots, x_m^{p^{n_m}}),$$ which is clearly of dimension $p^{\sum n_j}.$ This isomorphism will be induced from a surjective map $$\bar{\bar{f}}: k[x_1, \ldots, x_m]/(x_1^{p^{n_1}}-s_1, \ldots, x_m^{p^{n_m}}-s_m) \longrightarrow V$$ with all of the $s_i$ given by lemma \ref{WNkDivispN module}. We will consider a multi-filtration on $ k[x_1, \ldots, x_m]/(x_1^{p^{n_1}}-s_1, \ldots, x_m^{p^{n_m}}-s_m)$  induced from the one on $ k[x_1, \ldots, x_m]$ given by the lexicographic order on the degrees.  The associated graded is isomorphic to $k[x_1, \ldots, x_m]/(x_1^{p^{n_1}}, \ldots, x_m^{p^{n_m}})$ so this will induce a map
$$gr(\bar{\bar{f}}): k[x_1, \ldots, x_m]/(x_1^{p^{n_1}}, \ldots, x_m^{p^{n_m}}) \longrightarrow V$$
which will be shown to be an isomorphism.

By Lemma \ref{WNkDivispN module}, there exists a common null-vector $v$ to all the $D_{ij}$.  Set $V' = W_N \cdot v$ to be the $W_N$ submodule of $V$ generated by $v$.  

Consider $W_N \cdot b$, the one-dimensional free $W_N$ module generated by a symbol $b$.  Then, we have a map
$$f: W_N\cdot b \longrightarrow W_N\cdot v.$$
It is defined on $b$ by $f(b) := v$, and extended to $w\cdot b \in W_N \cdot b$ by $f(w\cdot b) = w\cdot f(b) = w\cdot v$.  This map is clearly surjective.  We want to show that $k[x_1, \ldots, x_m]/(x_1^{p^{n_1}}-s_1, \ldots, x_m^{p^{n_m}}-s_m) \cdot b$ is a quotient of $W_N \cdot b$ and that $f$ will induce the map $\bar{\bar{f}}$ that we are looking for.
More precisely we will show that $f$ produces a surjective module morphism
$$\bar{f}: k[x_1, \ldots, x_m]\cdot b\twoheadrightarrow W_N\cdot v,$$
which in turn will induce
$$\bar{\bar{f}}: k[x_1, \ldots, x_m]/(x_1^{p^{n_1}}-s_1, \ldots, x_m^{p^{n_m}}-s_m)\cdot b \longrightarrow W_N.v.$$

Let $(D_{i p^k})$ be the left ideal generated by all $D_{i p^k}$ for $0 \le k < n_i$ with $1\le i\le m$.  Then, $(D_{i p^k})\cdot b$ is a submodule of $W_N\cdot b.$  If we show that $(D_{i p^k}) \cdot b \subset Ker(f)$, then by Lemma \ref{Mapping between Spaces and Quotients}, there will be an induced map 
$$\bar{f}: W_N\cdot b / (D_{i p^k})\cdot b \twoheadrightarrow W_N\cdot v.$$
Since $W_N \cdot b / (D_{i p^k}) \cdot b \cong \left( W_N / (D_{i p^k}) \right) \cdot b \cong k[x_1, \ldots, x_m] \cdot b,$ we will have the desired map $\bar{f}: k[x_1, \ldots, x_m]\cdot b\longrightarrow W_N\cdot v.$

We therefore show that $(D_{i p^k}) \cdot b \subset Ker(f)$.  Consider an arbitrary element in $(D_{i p^k}) \cdot b$. It is of the form $\sum u_{ik}D_{ip^k} \cdot b$ for some $u_{ik} \in W_N$.  Since $f(\sum u_{ik}D_{ip^k} \cdot b)=\sum u_{ik}D_{ip^k}  \cdot f(b)=\sum u_{ik}D_{ip^k} \cdot v$, and since we have chosen $v$ so that $D_{ip^k} v=0,$ we are done.

It remains to show that this map
$$\bar{f}: k[x_1, \ldots, x_m]\cdot b\longrightarrow W_N\cdot v$$

that we have just built indeed descends to a map
$$\bar{\bar{f}}: k[x_1, \ldots, x_m]/(x_1^{p^{n_1}}-s_1, \ldots, x_m^{p^{n_m}}-s_m)\cdot b \longrightarrow W_N.v.$$
 
By Lemma \ref{WNkDivispN module}, $x_j^{p^{n_j}}(v) = s_j(v)$, which means that $(x_1^{p^{n_1}}-s_1, \ldots, x_m^{p^{n_m}}-s_m)\cdot b \subset Ker(\bar{f})$. This in turn produces a map $\bar{\bar{f}}: k[x_1, \ldots, x_m]/(x_1^{p^{n_1}}-s_1, \ldots, x_m^{p^{n_m}}-s_m) \cdot b \twoheadrightarrow W_N \cdot v$ by Lemma \ref{Mapping between Spaces and Quotients}. This map is clearly surjective. Since $V$ is irreducible, and $W_N \cdot v$ is a non zero submodule, one has $V\cong W_N \cdot v$ and hence $\bar{\bar{f}}$ is onto.

To show that $\bar{\bar{f}}$ is injective, it suffices to show that  $k[x]/(x^{p^n})$ is irreducible.  So, let $B$ be a non-trivial submodule of $k[x]/(x^{p^n})$,  we will show that it coincides with $k[x]/(x^{p^n})$.  We want to show that  B contains an element of the form 1+higher terms, since such an
element generates the module $k[x]/x^{p^n}$ over $k[x]$. Let $b$ be an arbitrary nonzero vector in B. Let $x^n$ be the lowest monomial it contains. (We normalize $b$ so that $b = x^n + \text{higher terms}$.) As 
$$D_{n} b = 1+\text{higher terms},$$
 we have that $B = k[x]/(x^{p^n})$.  $\square$
\end{proof}

\subsection{Applications} \label{app}

Suppose that we work over  an algebraically closed field $k$ of characteristic $p$.  Denote the algebra $A_n := k \langle x_1, x_2, \ldots, x_n \rangle$.  Our noncommutative algebra is $A := A_n / (f_1, f_2, \ldots, f_m)$, where each  relation $f_i$ is a noncommutative polynomials in $x_j^{p^{n_j}},$ for some integers $n_j$.

\begin{theorem} 
\label{WNkNiAaction}
In the above setting, the algebra $\bigotimes_j W_{n_j}(k)$ acts on $N_i(A)$.
\end{theorem}

Half of this action (cor. \ref{cly}) comes from the following proposition.

\begin{proposition} \label{action} The algebra $A_n/M_3$ acts on $N_i(A)$.
\end{proposition}

We will need twice the remark, true by induction, that a morphism of algebras $\phi: A\longrightarrow B$ preserves $N_i's$, i.e 
\begin{equation} \label{inclusionNi}
\phi(N_i(A))\subset N_i(B).
\end{equation}

\begin{proof} (of prop. \ref{action})
We first remark that $A/M_3$ acts on $N_i(A)$. Indeed, Cor. 1.4 in \cite{BaJ}
states that $$M_jM_k\subset M_{j+k-1}$$ whenever $j$ or $k$ is odd. In particular $M_3M_j\subset M_{j+2}\subset M_j,$ and the left multiplication by $M_3$ on $M_j$ preserves $M_{j+1}$. The left multiplication by $A$ induces therefore an action of $A/M_3$ on $N_i(A):=M_i(A)/M_{j+1}(A).$
 Now, since the natural projection $p:A_n\longrightarrow A$ is a map of algebras, the remark $(\ref{inclusionNi})$ above  applies and in particular $p(M_3(A_n))\subset M_3(A)$. This means that $p$ descends to a morphism $\bar{p} :A_n/M_3\longrightarrow A/M_3$. So the $A/M_3$-module $N_i(A)$ becomes an $A_n/M_3$-module by composition with $\bar{p}$. $\square$

\end{proof}

\begin{corollary}\label{cly} The polynomial algebra $k[x_1,\dots,x_n]$ acts on $N_i(A)$.
\end{corollary}

To prove this corollary, let us recall the following.

\begin{fact} (Prop. 3.1 in \cite{BEJ+})
The algebra $A_n/M_3$ is the algebra generated by $x_1,\dots, x_n$ and $u_{i,j}$ for $i,j\in [1,\dots,n]$ and $i\neq j$, with the following relations:\begin{enumerate}
\item[(1)] $u_{ij}=[x_i,x_j],$ and so $u_{ij}+u_{ji}=0$;
\item[(2)] $[x_i,u_{jk}]=0$ for all $i,j,k;$
\item[(3)] $u_{ij}$ commute with each other: $[u_{ij},u_{kl}]=0;$
\item[(4)] $u_{ij}u_{kl}=0$ if $i,j,k,l$ contains repetitions;
\item[(5)] $u_{ij}u_{kl}=-u_{i,k}u_{jl}$ if $i,j,k,l$ are all distincts.
\end{enumerate}
\end{fact}

\begin{proof} (of cor. \ref{cly}) By  (2) and (3) one has that $u_{ij}$ is central in $A_n/M_3$. So $u_{ij}$ acts as a scalar on irreducible modules (Schur's lemma). This scalar is $0$ since $u_{ij}^2=0$ by (4) above. This implies, using the Jordan-Holder decomposition of an arbitrary module, that $u_{ij}$ acts as $0$. But this means that this action descends to an action of $k[x_1,\dots,x_n]$. We can then apply this result to the action given by prop. \ref{action}.$\square$
\end{proof}

To understand the origin of the other half of the action of $\bigotimes_j W_{n_j}(k)$ on $N_i(A)$, namely the action of the $(D_m)_{x_i}$'s, let us remark that $D_m$ is the coefficient of $t^m$ in the expression of the automorphism $$T:=e^{tD}=\sum_m \frac{D^m}{m!}t^m$$ of the algebra $A_n \otimes k[t] / t^{p^n}$. This is convenient because it suffices to define the action of $T$  to deduce the action of the $D_m$'s. 

\begin{proof} (of thm.\ref{WNkNiAaction})
  Consider the automorphism $T_j:=e^{t\frac{\partial}{\partial x_j}}$ of the algebra $A_n \otimes k[t] / t^{p^n}.$ It is given on generators by 
$$T_j(x_i) := \begin{cases}
x_j + t & i = j \\ 
x_i & i\neq j.
\end{cases}
 $$
We first want to show that $T_j$ descends to an automorphism of $A \otimes k[t] / t^{p^n}.$ It suffices to check that $T_j(f)= f$ for $f$ a non commutative polynomials in $x_j^{p^{n_j}}$'s.  Since $f$ is of the form $f=\sum \Pi_j (\alpha_j x_1^{p^n}) \alpha_l,$ with $ \alpha_k\in k \langle  x_1, \ldots, x_n \rangle$ not containing $x_j$, one has that $T_j (\alpha_k)=\alpha_k$. So in particular
 $$T_j(f)=\sum \Pi_j  (\alpha_j T_j(x_1^{p^n})) \alpha_l = f,$$ since  $ T_j(x_j^{p^n}) = (T_j(x_j))^{p^n} = (x_j + t)^{p^n} = x_j ^{p^n} +t^{p^n}=x_j ^{p^n}.$

We can now apply the remark $(\ref{inclusionNi})$ to $T_j$ to obtain the action of  $T_j$ on $N_i(A)$. We then define the action of $(D_k)_{x_j}$ on $N_i(A)$, for $k< p^{n_j}$, to be the coefficient of $t^k$ in the representation of $T_j$ on $N_i(A)$, i.e.  $$T_j =: \sum t^k (D_k)_{x_j}.$$  
$\square$
\end{proof}

\begin{corollary}
\label{NiDivisPN}
With the conditions of Theorem \ref{WNkNiAaction}, if $N_i(A)$ is finite dimensional (i.e. if the abelianization $A_{ab}$ is finite dimensional), and if the relations are noncommutative polynomials in the variables $x_1^{p^{n_1}}, \ldots, x_m^{p^{n_m}})$, then $\text{dim}(N_i(A))$ is divisible by $p^{\sum{n_i}}$.
\end{corollary}

\begin{proof}
According to Proposition \ref{WNkDivispN}, each finite dimensional representation of $W_n(k)$ has dimension divisible by $p^n$.  In the case of the relations being polynomials of $x_i^{p^{n_i}}$ with $1 \le i \le r$, the tensor product algebra $\bigotimes_i W_{n_i}(k)$ acts on $N_i(A)$.  Because this is a tensor product of irreducible representations of $W_{n_i}(k)$, each of its irreducible representations has dimension divisible by $p^{\sum n_i}$.   $\square$
 
\end{proof}

\begin{corollary}
\label{HilbertSeriesDivisN}
Except for finite dimensionality of $N_i(A)$, suppose that in the situation of Corollary \ref{NiDivisPN}, the relations are homogeneous in $x_1, \ldots, x_r$.  Then, the Hilbert series of $N_i(A)$ with respect to the corresponding variables $X_1, \ldots, X_r$ is divisible by 
$$(1+X_1 + \cdots + X_1^{p^{n_1} - 1})\ldots(1+X_r + \cdots + X_r^{p^{n_r} - 1}),$$
 in the sense that the ratio is a power series with non-negative integer coefficients.
\end{corollary}

\begin{proof}
Consider the case $r = 1$, as the general proof follows similarly.  Let $M = N_i (A)$.  It is a $\mathbb{Z}$-graded module over $W_n(k)$, with a grading given by $deg(x) = 1$, $deg(D) = -1$, and nonnegative degrees of the vectors.
Because of this, we may take any homogeneous vector and apply $D_j$ until getting $0$; thus, there exists a common null vector of $D_j$, namely $v_1 \ne 0$.
Let $M_1 = F_1$ be the submodule generated by $v_1$, then it has a basis of $\langle v_1, xv_1, x^2 v_1\ldots \rangle$.
Thus, we have two cases for $M_1$.  First, if none of these $x^{s} v_1 = 0$, then $M_1 \cong k[x]$.  Second, if $x^s v_1 = 0$, where $s$ is minimal but positive, then we have that $s$ is a multiple of $p^{n_1}$ as by Theorem \ref{WNkNiAaction}.  Thus, $M_1 \cong k[x]/(x^{jp^{n_1}})$ for some positive integer $j$.

Next, let $v_2 \ne 0$ be a common null vector of $D_j \in M/F_1$.  We define $M_2$ as the submodule in $M/F_1$ generated by $v_2$, and $F_2$ as the preimage of $M_2$ in $M$.  Continuing this construction, we make an exhaustive filtration $F_1 \subset F_2 \subset F_3 \subset \ldots$ of $M$ such that $F_i / F_{i-1} = M_i$, and all $M_i \cong k[x]$ or $k[x]/(x^{jp^{n_1}})$.

If $E$ is a graded vector space, denote $h_E$ as the Hilbert Series of $E$, i.e., if $E = \bigoplus_i E_i$, then $h_E = \sum_i \text{dim}(E_i) X^i$.

Since $h_{N_i(A)} = h_{M_1} + h_{M_2} + \cdots$, we are done if each $h_{M_i}$ is divisible by the desired polynomial.  To this end, note that if $M_i  \cong k[x] \cdot v = \langle v, xv, \ldots \rangle$ and $deg(v) = \ell$, then $h_{M_i} = X^{\ell} + X^{\ell+1} + \cdots = (1+X+ \cdots + X^{p^{n_1}-1})(X^{\ell} + X^{\ell + p^{n_1}} + \cdots)$.  And, if $M_i \cong k[x]/(x^{p^{n_1} j}) \cdot v'$, where $deg(v') = \ell'$, then $h_{M_i} = X^{\ell'} + X^{\ell' +1} + \cdots +X^{\ell' + (j-1)p^{n_1}} = (1+X + \cdots + X^{p^{n_1} - 1})(X^{\ell'} + X^{\ell' + p^{n_1}} + \cdots + X^{\ell' + (j-1)p^{n_1}})$.  $\square$
\end{proof}


\section{Bigraded Structure of $N_2$ and $N_3$ over $\mathbb{Z}$}
\label{Z section}

In this section, we give complete descriptions of the abelian group of $N_i(A)$ for $i = 2,3$ and $A = A_2 /(x_1^m, x_2^n)$, where $A_2 = \mathbb{Z} \langle x_1, x_2 \rangle$.
A bigrading of $A_k$, the free algebra with $k$ generators, is given by the total degree in $x_1, x_2, \ldots, x_k$.  This gives us more information about the inherent structure of the algebra.

However, with the added relations from the ideal $(x_1^m, x_2^n)$, which is generated by homogeneous terms in $x_1, x_2$, $A$ inherits a bigrading from $A_2$ which is bounded by $_m,n)$.  More precisely, the bigrading of a monomial $P$ is given by $(|P|_{x_1}, |P|_{x_2})$, where $|P|_{x_1}$ denotes the total degree in $x_1$ of $P$ and $|P|_{x_2}$ denotes the total degree in $x_2$ of $P$.  For example, the bigrading of the term $x_1 x_2^3 x_1$ is given by $(2,3)$.

In fact, the bigrading over $A_2$ and $A$ induce a grading over $N_2 (A)$ and $N_3 (A)$.  

When we view $N_i(A_2)$ as finite-dimensional abelian groups, we may induce a bigrading based upon the degrees of each generator.

Since these are abelian groups, they may be decomposed into a free part (copies of $\mathbb{Z}$) and a torsion part (direct sum of $\mathbb{Z}_m$ for integral $m$) by the Fundamental Theorem of Finitely Generated Abelian Groups.  Thus, using the data generated by our {\it MAGMA} program, we conjecture and prove the structures of $N_2$ and $N_3$.

We will use the simple but well-known Leibniz Rule throughout:
\begin{lemma}
\label{Leibniz Rule}
$$[a_1 \ldots a_n, b] = \sum_{i = 1}^n a_1 \ldots a_{i-1} [a_i, b] a_{i+1} \ldots a_n$$
and
$$[a, b_1 \ldots b_n] = \sum_{i = 1}^n b_1 \ldots b_{i-1} [a, b_i] b_{i+1} \ldots b_n.$$
\end{lemma}

\subsection{Structure of $N_2$}\label{N2}

The aim of this section is to show that the abelian group structure of $N_2$ is given by the following table:

\begin{table} [H] 
\small
\caption{Bigraded Description of $N_2 (A)$}
\label{Bigraded Description of $N_2 (A)$}
\begin{tabular}{|l||l|l|l|ll|l|l|}
\hline
$(m,n)$ & $0$ & $1$ & $2$ & $\cdots$ & $\cdots$ & $n-1$ & $n$ \\
\hline
\hline
$0$ & $0$ & $\cdots$ & $\cdots$ & $\cdots$ & $\cdots$ & $\cdots$ & $\cdots$  \\
\hline
$1$ & $\vdots$ & $\mathbb{Z}$ & $\mathbb{Z}$ & $\cdots$ & $\cdots$ & $\mathbb{Z}$ & $\mathbb{Z}_n$ \\
\hline
$2$ & $\vdots$ & $\mathbb{Z}$ & $\mathbb{Z}$ &$\cdots$ & $\cdots$ & $\mathbb{Z}$ & $\mathbb{Z}_n$  \\
\hline
$\ \vdots$ & $\vdots$ & $\vdots$ & $\vdots$ & $\ddots$ & & $\vdots$ & $\vdots$  \\

$\ \vdots$ & $\vdots$ & $\vdots$ & $\vdots$ & & $\ddots$ & $\vdots$ & $\vdots$   \\
\hline
$m-1$ & $\vdots$ & $\mathbb{Z}$ & $\mathbb{Z}$ & $\cdots$ & $\cdots$ & $\mathbb{Z}$ & $\mathbb{Z}_n$\\ \hline
$m$ & $\vdots$ & $\mathbb{Z}_m$ & $\mathbb{Z}_m$ & $\cdots$ & $\cdots$ & $\mathbb{Z}_m$ & $\mathbb{Z}_{(m,n)}$  \\
\hline

\end{tabular}
\end{table}
where $(m,n) = gcd(m,n)$.
In other terms we want to show that

\begin{theorem}
\label{N_2 Basis over Z}
The free part of $N_2 (A)$ as a $\mathbb{Z}$-module has a basis $\{ x_1^ix_2^jy \mid 0 \leq i \leq m-1, 0\leq j \leq n-1 \}$.  (Free part description)
\end{theorem}

and that 

\begin{theorem}
\label{N_2 Torsion over Z}
As a $\mathbb{Z}$-module, the elements $x_1^ix_2^{n-1}y$ for $0 \le i \le m-2$ (resp $x_1^{m-1}x_2^jy$) are of torsion of order $n$ (resp $m$), except when $i =m-1$ for which $x_1^{m-1}x_2^{n-1}y$ is of order $(m,n)$.  (Torsion part description)
\end{theorem}

Our chain of reasoning in proving Theorem \ref{N_2 Basis over Z} starts with forming a basis of $M_2(A_2)$ (Lemma \ref{Basis of M_2}). This induces a generating family of $N_2(A_2)=M_2(A_2)/M_3(A_2)$ with eventually some redundancy.  In order to eliminate this redundancy, we will rewrite these elements using $R$ to arrive to a normal form and obtain a basis of $N_2(A_2)$.  Finally, if we take into account the extra relations of $A$ to find as basis of $N_2 (A)$ (Theorem \ref{N_2 Basis over Z}), then some torsion appears.  This torsion part induced by the relations will be separated from the free part of $N_2(A)$.

Let us recall a presentation of $A/M_3$ from \cite{BEJ+}, inspired by the seminal paper \cite{FS} by Feigin and Shoikhet.

\begin{theorem}
\label{BEJ+ A_2/M_3}
$A_2/M_3 = \langle x_1,x_2,y \rangle /(R)$ where $R$ is the set of relations 

\begin{equation}
\label{BEJ rel 1}
[x_1,x_2]=y,
\end{equation}
\begin{equation}
\label{BEJ rel 2}
[x_1,y]=[x_2,y]=y^2=0.
\end{equation}
\end{theorem}

Below are three tools we use intermediately to prove \ref{N_2 Basis over Z}:

\begin{remark}
\label{Basis of A_n}
A basis of $A_n$, denoted $\mathcal{B}(A_n)$, is given by monomials in the generators $x_1, \cdots, x_n$.
\end{remark}

\begin{lemma}
\label{Basis of M_2}
A basis of $M_2(A_2)$
is given by $\{ v y w \mid v,w \in \mathcal{B}(A_n) \}$.
\end{lemma}

\begin{proposition}
\label{Basis of N_2(A_2)}
A basis of $N_2(A_2)$ as a $\mathbb{Z}$-module is given by $\{x_1^{i} x_2^{j} y \}$.
\end{proposition}

\begin{proof}
Recall the definition of $M_2(A_2) = A_2 L_2(A_2) A_2 = A_2[A_2,A_2] A_2.$   Any element in this $M_2$ is a linear combination of $u[v,w]z$ where $u,v,w,z \in \mathcal{B}(A_2)$.

The Leibniz rule gives $u[v,w]z = u(\sum v_i y w_i)z$ for some $v_i, w_i \in \mathcal{B}(A_2)$; note this means that there is at least one $y$ term in each monomial, and $M_2$ is spanned by $\{v' y w' \mid v', w' \in \mathcal{B}(A_2) \}$.
It is simply a routine checking to verify the linear independence of this basis.  $\square$

\end{proof}

We now prove Proposition \ref{Basis of N_2(A_2)}.

\begin{proof}
Starting with a basis $\mathcal{B}_2$ of $M_2(A_2)$ given by Theorem \ref{BEJ+ A_2/M_3}, we use the relations from Theorem \ref{N_2 Basis over Z} to rewrite the elements of its image $\bar{\mathcal{B}_2}$ in $N_2(A_2) = M_2 (A_2) / M_3(A_2)$ in a normal form.  Using relation (2), we may commute $y$ anywhere within, so we push them to the right of every term by convention.  

We now show that $x_1$ and $x_2$ commute in monomials which contain a $y$.  Let $u,w \in \mathcal{B}(A_2)$:
$$u x_2 x_1 w y \stackrel{(\ref{BEJ rel 1})}{=} u (x_1 x_2 - y) w y = u x_1 x_2 w y - u y^2 \stackrel{(\ref{BEJ rel 1})}{=} ux_1 x_2 w y.$$
Thus, any element of $\bar{\mathcal{B}_2}$ may be rewritten in the form of $x_1^i x_2^j y$; the set of all such elements is still a generating family, but now is linearly independent in the quotient. $\square$
\end{proof}

These set us up for the proof of Theorem \ref{N_2 Basis over Z}.
\begin{proof}

Now, we finally work with $N_2(A)$.
To show that $0 \le i \le m-1$ and $0 \le j \le n-1$, recall that $A$ has the additional relations $x_1^m, x_2^n$, so if $i \ge m$ or $j \ge n$, then $x_1^i x_2^j y$ vanishes.
But, for $0 \le i \le m-1$ and $0 \le j \le n-1$, no torsion can occur in total degree $i+j < m$ or $n$.
This is because $m$ and $n$ are the degrees of $A$'s relations.  $\square$
\end{proof}

And Theorem \ref{N_2 Torsion over Z}:

\begin{proof}
For each bidegree with torsion, we specifically calculate the terms causing torsion.  For example, to find those with bidegree $(m,1)$, we first note that the term must be of the form $x_1^m y$ by Proposition \ref{Basis of N_2(A_2)}.  The generators of $N_2(A)$ are the images of the generators of $N_2(A_2)$ modulo relations $x_1^m = 0, x_2^n = 0$.  To show its torsion, note that:
$$0 = [x_1^m, x_2] = \sum_{s = 0}^{m-1} x_1^s [x_1,x_2] x_1^{m-s-1} = \sum_{s= 0}^{m-1} x_1^s y x_1^{m-s-1} = \sum_{s=0}^{m-1} x_1^{m-1} y = mx_1^{m-1}y.$$

Similarly, we find that $nx_2^{n-1}y = 0.$

Thus, for all $j < n$, we have that $mx_1^{m-1}x_2^{j}y = 0$, so there is $\mathbb{Z}_m$ torsion there.  Likewise, we find $nx_1^{i} x_2^{n-1} y = 0$ for $i < m$, so there is $\mathbb{Z}_n$ torsion there.

However, let us consider what happens with $x_1^{m-1}x_2^{n-1}y$.  We know that $mx_1^{m-1}x_2^{n-1}y = nx_1^{m-1}x_2^{n-1}y = 0$.  Let $k$ be the order of $x_1^{m-1}x_2^{n-1}y$; then, since for all $a,b$, $amx_1^{m-1} x_2^{n-1}y = bnx_1^{m-1}x_2^{n-1}y = 0$, by Bezout's Lemma we have $k \mid (m,n)$.
Thus, the term generates the group $\mathbb{Z}_{(m,n)}$.   $\square$
\end{proof}

\subsection{Structure of $N_3$}
In this section, we prove that the non-zero terms in the bigraded structure of $N_3$ are given by the following table:

\begin{table} [H] 
\small
\caption{Bigraded Description of $N_3 (A)$}
\label{Bigraded Description of $N_3 (A)$}
\begin{tabular}{|l||l|l|l|ll|l|l|l|}
\hline
$(m,n)$ & $0$ & $1$ & $2$ & $\cdots$ &  $\cdots$ & $n-1$ & $n$ & $n+1$  \\
\hline
\hline
$0$ & $0$ & $\cdots$ & $\cdots$ & $\cdots$ & $\cdots$ & $\cdots$ & $\cdots$ & $\cdots$ \\
\hline
$1$ & $\vdots$ & $0$ & $\mathbb{Z}$ & $\cdots$ & $\cdots$ & $\mathbb{Z}$ & $\mathbb{Z}$ & $\mathbb{Z}_{f(n)}$\\
\hline
$2$ & $\vdots$ & $\mathbb{Z}$ & $\mathbb{Z}^3$ & $\cdots$ & $\cdots$ & $\mathbb{Z}^3$ & $\mathbb{Z}^2 \oplus \mathbb{Z}_n$ & $\mathbb{Z}_n \oplus \mathbb{Z}_{f(n)}$  \\
\hline
$\ \vdots$ & $\vdots$ & $\vdots$ & $\vdots$ & $\ddots$ & $$ & $\vdots$ & $\vdots$ & $\vdots$  \\

$\ \vdots$ & $\vdots$ & $\vdots$ & $\vdots$ & $$ & $\ddots$ & $\vdots$ & $\vdots$ & $\vdots$ \\
\hline
$m-1$ & $\vdots$ & $\mathbb{Z} $ & $\mathbb{Z}^3$ & $\cdots$ & $\cdots$ & $\mathbb{Z}^3$ & $\mathbb{Z}^2 \oplus \mathbb{Z}_n$ & $\mathbb{Z}_n \oplus \mathbb{Z}_{f(n)}$ \\ \hline
$m$ & $\vdots$ & $\mathbb{Z}$ & $\mathbb{Z}^2 \oplus \mathbb{Z}_m$ & $\cdots$ & $\cdots$ & $\mathbb{Z}^2 \oplus \mathbb{Z}_m$ & $\mathbb{Z}_m \oplus \mathbb{Z}_n$ & $\mathbb{Z}_{f(n)} \oplus \mathbb{Z}_{(m,n)}$\\
\hline
$m+1$ & $\vdots$ & $\mathbb{Z}_{f(m)}$ & $\mathbb{Z}_m \oplus \mathbb{Z}_{f(m)}$ & $\cdots$ & $\cdots$ & $\mathbb{Z}_m \oplus \mathbb{Z}_{f(m)}$ & $\mathbb{Z}_{f(m)} \oplus \mathbb{Z}_{(m,n)}$ & $\mathbb{Z}_{(m,n)}$ \\
\hline
\end{tabular}
\end{table}

Where $(m,n) = gcd(m,n)$ and
\begin{definition}
\label{gcdDef}
The function $f: \mathbb{N} \longrightarrow \mathbb{N}$ is defined by
$$f(k) := \begin{cases}
k & k \text{ odd} \\ 
\dfrac{k}{2} & k \text{ even}.
\end{cases}
 $$
\end{definition}
In addition to the notation $y := [x_1, x_2]$ from the previous section, we introduce the following two terms: $z_1 := [x_1, y]$, $z_2 := [x_2, y].$

In this section, we prove the following lemmas about the structure of  $N_3(A)$ using tools similar to those from the previous section:

\begin{lemma}
\label{N_3 Basis}
The free part of $N_3$ is generated as a $\mathbb{Z}$-module by the following terms: $x_1^i x_2^j z_1$, $x_1^i x_2^j z_2$, $x_1^i x_2^j y^2$, for $0 \le i \le m-1$, $0 \le j \le n-1$.
(Free part description)
\end{lemma}

\begin{lemma}
\label{N_3 Torsion}
As a $\mathbb{Z}$-module, the $x_1^{m-1} x_2^{j}y^2$ and $x_1^{m-1} x_2^{j+1} z_1$ (resp $x_1^i x_2^{n-1} y^2$ and $x_1^{i+1} x_2^{m-1} z_2$) terms are of torsion of order $m$ and $f(m)$ (resp $ n $ and $ f(n) $), except when $j=n-1$, for which $x_1^{m-1}x_2^{n-1}y$ is of order $(m,n)$.  (Torsion part description)

\end{lemma}

First, we show that 

\begin{proposition}
\label{M_3 Basis}
$M_3$ is generated by $u [x_1, y] v$, $u [x_2, y] v$, and $uyvyw$, for $u, v, w \in A$.
\end{proposition}

\begin{proof}
We first show that $M_3$ is generated by $u [g, y] v$ and $uyvyw$, for $u, v, w \in A$ and $g \in \{x_1, x_2\}$.
By definition, $M_3 = A [A,[A,A]] A$, so any of its elements may be written as $u [a,[b,c]]v$ for some $u,a,b,c,v$ in $A_2$.

We will concentrate on showing that $[a,[b,c]]$ can be written as a sum of $u[g, [b,c]]u'$.
Consider $a = a_1 \cdots a_k$, where each of $a_i \in \{x_1, x_2\}$.  We are done if we use the Leibniz Rule:
$$[a,[b,c]] = [a_1 \cdots a_k, [b,c]] = \sum_{i = 1}^k a_1 \cdots a_{i-1} [a_i, [b,c]] a_{i+1} \cdots a_k.$$

Next, we will show that $[g,[b,c]]$ can be written as a sum of $u[g,[g',d]]v$, with $g,g' \in \{x_1, x_2\}$ and $u,d,v \in A$.  We apply the Jacobi identity to get $[g,[b,c]] = [b,[g,c]] - [c, [g,b]].$
Looking at the first term $[b,[g,c]]$, we can apply the Leibniz rule as before to show that it can be written as a sum of $u[g,[g',d]]v$.  Since the second term is the same up to order as the first term, we are done.

Finally, we consider terms of the form $[g,[g',d]]$, showing that they can be written as a sum of the desired basis terms of $u[g,y]v$ and $uyvyw$.  
Let $d = d_1 \cdots d_j$, with each $d_i \in \{x_1, x_2\}$.  We apply the Leibniz rule once again, this time to $d$.  Thus,
$$[g,[g',d]] = \sum_{i = 1}^j [g,u_i [g', d_i] v_i] = \sum_{i=1}^j (u_i [g,[g',d_i]]v_i + u_i[g',d_i][g,v_i] +  [g,u_i][g',d_i]v_i)$$
for some $u_i, v_i \in A$.  The term $u_i [g,[g',d_i]]v_i$ is of the form of $u [g,y] v$ already, as $[g',d_i] = y \text{ or } 0$.
To show that $[g,v_i]$ (and simultaneously $[g,u_i]$) is in the form of $uyw$ (or is equal to $0$) with $u,w \in A$, we apply the Leibniz rule to $v_i = v_{i,1} \cdots v_{i ,\ell}$ with $v_{i,j} \in \{x_1, x_2\}$.
$$u_i y [g,v_i] = u_i y \sum_{j = 1}^{\ell}  v_{i,1} \cdots v_{i,j-1} [g, v_{i,j}] v_{i,j+1} \cdots v_{i,\ell} = \mathop{\sum_{j=1}^{\ell}}_{v_{i,j} \ne g} u_i y w_j y v_j$$
with $w_j, v_j \in A$, which completes the proof.  $\square$
\end{proof}

Then, we recall a theorem by \cite{EKM} that is the $M_4$ analogue of Theorem \ref{BEJ+ A_2/M_3}.

\begin{theorem}
\label{EKM A_2/M_4}
A presentation of $A_2 / M_4$ is given by the generators $x_1, x_2$  the following relations:
$$[x_1, z_2] = [x_1, z_1] = [x_2, z_1] = [x_2, z_2] = 0, \ \ y z_1 = y z_2 = y^3 = 0,\ \  z_1^2 = z_1 z_2 = z_2^2 = 0.$$
\end{theorem}

Armed with Lemma \ref{M_3 Basis} and Theorem \ref{EKM A_2/M_4}, we can find a basis of $M_3 / M_4 = N_3$. 

\begin{proof}
Our aim is to rewrite the terms $E$ and $F$ in a normal form using rewriting rules from $M_4$'s basis, where $E := u [x, y] v$ and $F:= uyvyw$ for $u,v,w \in A$ and $x \in \{x_1, x_2\}$.

Using methods similar to those in Theorem \ref{N_2 Basis over Z}, we find that $x_1$ and $x_2$ in monomials like $E$ commute, and so $E = x_1^i x_2^j z_1$ or $x_1^i x_2^j z_2$.

Next, we will rewrite $F$.  We first note that if there is more than one $y$ present in any monomial, then all the $y's$ commute with everything within that term, so $F$ may be rewritten as $uvw y^2$.  Like previously, if $F = uvw y^2 \ne x_1 ^i x_2 ^j y^2$, then we can also commute each $x_1$ and $x_2$ in these terms.  $\square$

\end{proof}

We will use a fact in the proof of Lemma \ref{N_3 Torsion}:
\begin{proposition}
\label{xiy comm}
Let $i \ge 1$.  Then, $yx_1^i = x_1^i y - ix_1^{i-1} z_1$ and $yx_2^i = x_2^i y - i x_2^{i-1} z_2$.
\end{proposition}

\begin{proof}

To find the torsion, we identify all relations for bidegree $(m+1,2)$, and work our way up from there.

We start off with some algebraic manipulation to get that
\begin{equation}
0 = mx_1^{m-1}y + \frac{m(m-1)}{2} x_1^{m-2} z_1.
\end{equation}
Let $E$ be the right hand of equation $(1)$.

First, we would like to prove $m(m-1) x_1^{m-1} y^2 = 0.$  Starting with $0 = [E, x_2]$, we get that 
$$ 0 = m(m-1)x_1^{m-2}y^2. $$
Multiplying on the right by $x_1$ yields our first relation.

Second, we would like to show that $mx_1^{m-1}y^2 = 0.$  Right multiplication on equation (1) by $y$ yields the relation.

Third, we will show that $m x_1^{m-1} x_2 z_1 = 0$.  With right multiplication by $x_2$ on the equation $mx_1^{m-1}z_1 = 0$, commutativity of $z_1$ with everything yields our desired relation.

Finally, we will show that $\dfrac{m(m-1)}{2} x_1^{m-1} x_2 z_1 = 0$.  If we right multiply equation (1) by $x_1$, we get the following:
$$0 = \frac{m(m-1)}{2} x_1^{m-1} z_1. $$

To finish, we right multiply by $x_2$.  

Notice that these monomials end with either $y^2 \text{ or } z_1$, which both commute with $x_2$; thus, if we right multiply by $x_2^{j}$ for $0 \le j \le n-3$ we get our desired results.

So, we have found two terms that generate groups: $x_1^{m-1} x_2^{j}y^2$, and $x_1^{m-1} x_2^{j+1} z_1$, both with bidegrees $(m+1, j+2)$.  The first term generates a torsion part of order $gcd(m, m(m-1)) = m$, while the second generates a torsion part of order $(m, \frac{m(m-1)}{2})$.  Thus, the torsion in the bidegree is $\mathbb{Z}_m \oplus \mathbb{Z}_{(m, \frac{m(m-1)}{2})}$.  For odd $m$, this is equal to $\mathbb{Z}_m \oplus \mathbb{Z}_m$, and $\mathbb{Z}_m \oplus \mathbb{Z}_{\frac{m}{2}}$ for even $m$, so our prior Definition \ref{gcdDef} of $f(k)$ holds.

Since $x_1$ is symmetric with respect to $x_2$, we obtain the same results for the bidegrees $(i+2,n+1)$ for $0 \le i \le m-3$; i.e., the torsion is $\mathbb{Z}_n \oplus \mathbb{Z}_{(n, \frac{n(n-1)}{2})}$.
$\square$
\end{proof}

\section{Conclusion}
	In this project, we programmed {\it MAGMA} \cite{BCP} to compute data about the dimensions and ranks of these lower central series ideal quotients for various algebras.  Using this data, we formulated and proved conjectures concerning these quotients $N_i(A)$.  Just like how knowing sufficiently the divisors of an integer, we have proven a partial result about the substructure of an infinite and complex family of algebras in {\bf Section \ref{F_p section}}.  And, in {\bf Section \ref{Z section}} we characterized the bigraded structure of $N_2(A)$ and $N_3(A)$ for algebras with two generators over $\mathbb{Z}$.  In addition, we have gathered over $250$ bigraded tables and nearly $100$ totally graded tables, which can aid further exploration of these algebraic structures and applications.  
	
\section{Future Work}
There is still much that may be explored in this topic.  Over $\mathbb{Z}$, we could describe the bigraded structure of $N_4 (A_2)$ by utilizing a recently published paper by \cite{CK} that outlined a basis of $A/M_5$.  In addition, we could try to produce code and explore individual grading of more than just $2$ variables.  In general, we would like to be able to describe $N_i(A)$, where $A \cong \mathbb{Z} \langle x_1, \cdots, x_k \rangle / (x_1^{m_1}, \cdots, x_k^{m_k})$.
Potential further work is to perform individual grading on the $B_i(A)$ defined in the introduction.

There are several conjectures we were not able to prove by the time of submission:
\begin{enumerate}
\item By comparing Tables $5$ and $6$ in Section $5$, where the only difference is that they were calculated over $\mathbb{Z}$ versus $\mathbb{F}_p$, we seem to be able to recover Table $6$'s data from the others'.  We mod out the free parts by $p\mathbb{Z}$, leaving a copy of $\mathbb{Z}_p$.  If there was torsion $\mathbb{Z}_m$ in the Table over $\mathbb{Z}$, then there would be a copy of $\mathbb{Z}_{(p,m)}$ over $\mathbb{F}_p$.  All our tables support this.
\item We have a conjecture about generators for the free part of $N_4(A)$, that they are $x_1^i, x_2^j v$, where $v \in A$ has one of the following bidegrees: $(1,3), (2,2), (3,1), (2,3), (3,2), (3,3)$.  
\item Though we have a complete description of $N_2 (A)$, with $A = A_2 / (x_1^m, x_2^n)$, we have found proofs of the same fashion that allow us to conjecture the number of generating terms there are in the basis $N_2(A_k)$:
$$\sum_{i = 1}^{\lfloor \frac{k}{2} \rfloor} \dbinom{k}{2i}.$$
By using a complex filtration, a closed form of this expression can be found:
$$Re((1+i)^k).$$

\end{enumerate}


\section{Methods and Tables}
In order to calculate free and torsion subgroups of $N_i(A)$, we use preexisting code that calculated $N_i(A)$ over $\mathbb{Q}$ for one relation.  This required us to modify the code to allow for multiple relations, calculations over $\mathbb{Z}$ and $\mathbb{F}_p$, and most importantly: to calculate bigraded data (that is,  degrees of individual generators in $A_2$).  The code computes each $N_i$ after computing the corresponding $L_i$ and $M_i$, then moves on to the subsequent $N_{i+1}$.

However, computers can only handle linear systems of size a few thousands.  The dimension of $A_2$ in degree $n$ is $2^n$, so to compute data with degree $n$ about $N_i(A_2)$, we need to solve linear systems of size $2^n$.  Realistically, our last calculable value $n = 12$, as $2^{12} = 4096$ bigraded entries.  So, we work with many data tables of $N_i(A)$ for small $i < 12$, automating the collection process by writing {\it Java} and {\it BASH} scripts to convert data to {\it LaTeX} tables.  Below, we present a small selection of our data collection, which contains over $350$ tables.

The rows represent $m$ and the columns represent $n$, where our relations are $x_1^m  = x_2^n = 0$.  A cell with a small {\tiny $0$} represents no free component there, while a blank cell indicates that the computer was not able to calculate data there.  Each non-trivial cell is of the form $R, (T)$, where $R$ represents the rank of the free component $(\mathbb{Z}^R)$, while $(T)$, in parentheses, represents the torsion structure.  For example, in $(2,5)$ of Table \ref{refExmpTable}, $T = (2 \cdot 4)$ represents $\mathbb{Z}_2 \oplus \mathbb{Z}_{4}$.  Absence of parentheses indicates an absence of torsion.

\begin{table} [H] 
\footnotesize
\caption{$N_2:\ \mathbb{Z} \langle x_1,x_2\rangle /(x_1^3,x_2^7)$, Time: 906.16 sec, Memory: 780.78MB}
\begin{tabular} {|l||l|l|l|l|l|l|l|l|l|l|l|l|}
\hline
$(m,n)$ & $0$ & $1$ & $2$ & $3$ & $4$ & $5$ & $6$ & $7$ & $8$ & $9$ & $10$ & $11$ \\ \hline \hline
$0$ &{\tiny $0$ } & {\tiny $0$ } & {\tiny $0$ } & {\tiny $0$ } & {\tiny $0$ } & {\tiny $0$ } & {\tiny $0$ } & {\tiny $0$ } & {\tiny $0$ } & {\tiny $0$ } & {\tiny $0$ } & {\tiny $0$ } \\ \hline
$1$ &{\tiny $0$ } & $1$ & $1$ & $1$ & $1$ & $1$ & $1$ & {\tiny $0$ }{  $(7)$} & {\tiny $0$ } & {\tiny $0$ } & {\tiny $0$ } &  \\ \hline
$2$ &{\tiny $0$ } & $1$ & $1$ & $1$ & $1$ & $1$ & $1$ & {\tiny $0$ }{  $(7)$} & {\tiny $0$ } & {\tiny $0$ } &  &  \\ \hline
$3$ &{\tiny $0$ } & {\tiny $0$ }{  $(3)$} & {\tiny $0$ }{  $(3)$} & {\tiny $0$ }{  $(3)$} & {\tiny $0$ }{  $(3)$} & {\tiny $0$ }{  $(3)$} & {\tiny $0$ }{  $(3)$} & {\tiny $0$ } & {\tiny $0$ } &  &  &  \\ \hline
$4$ &{\tiny $0$ } & {\tiny $0$ } & {\tiny $0$ } & {\tiny $0$ } & {\tiny $0$ } & {\tiny $0$ } & {\tiny $0$ } & {\tiny $0$ } &  &  &  &  \\ \hline
$5$ &{\tiny $0$ } & {\tiny $0$ } & {\tiny $0$ } & {\tiny $0$ } & {\tiny $0$ } & {\tiny $0$ } & {\tiny $0$ } &  &  &  &  &  \\ \hline
$6$ &{\tiny $0$ } & {\tiny $0$ } & {\tiny $0$ } & {\tiny $0$ } & {\tiny $0$ } & {\tiny $0$ } &  &  &  &  &  &  \\ \hline
$7$ &{\tiny $0$ } & {\tiny $0$ } & {\tiny $0$ } & {\tiny $0$ } & {\tiny $0$ } &  &  &  &  &  &  &  \\ \hline
$8$ &{\tiny $0$ } & {\tiny $0$ } & {\tiny $0$ } & {\tiny $0$ } &  &  &  &  &  &  &  &  \\ \hline
$9$ &{\tiny $0$ } & {\tiny $0$ } & {\tiny $0$ } &  &  &  &  &  &  &  &  &  \\ \hline
$10$ &{\tiny $0$ } & {\tiny $0$ } &  &  &  &  &  &  &  &  &  &  \\ \hline
$11$ &{\tiny $0$ } &  &  &  &  &  &  &  &  &  &  &  \\ \hline
\end{tabular}
\end{table}

\begin{table} [H] 
\label{refExmpTable}
\footnotesize
\caption{$N_2:\ \mathbb{Z} \langle x_1,x_2 \rangle /(x_1^4,x_2^6)$, Time: 911.82 sec, Memory: 769.03MB}
\begin{tabular} {|l||l|l|l|l|l|l|l|l|l|l|l|l|}
\hline
$(m,n)$ & $0$ & $1$ & $2$ & $3$ & $4$ & $5$ & $6$ & $7$ & $8$ & $9$ & $10$ & $11$ \\ \hline \hline
$0$ &{\tiny $0$ } & {\tiny $0$ } & {\tiny $0$ } & {\tiny $0$ } & {\tiny $0$ } & {\tiny $0$ } & {\tiny $0$ } & {\tiny $0$ } & {\tiny $0$ } & {\tiny $0$ } & {\tiny $0$ } & {\tiny $0$ } \\ \hline
$1$ &{\tiny $0$ } & $1$ & $1$ & $1$ & $1$ & $1$ & {\tiny $0$ }{  $(2 \cdot 3)$} & {\tiny $0$ } & {\tiny $0$ } & {\tiny $0$ } & {\tiny $0$ } &  \\ \hline
$2$ &{\tiny $0$ } & $1$ & $1$ & $1$ & $1$ & $1$ & {\tiny $0$ }{  $(2 \cdot 3)$} & {\tiny $0$ } & {\tiny $0$ } & {\tiny $0$ } &  &  \\ \hline
$3$ &{\tiny $0$ } & $1$ & $1$ & $1$ & $1$ & $1$ & {\tiny $0$ }{  $(2 \cdot 3)$} & {\tiny $0$ } & {\tiny $0$ } &  &  &  \\ \hline
$4$ &{\tiny $0$ } & {\tiny $0$ }{  $(4)$} & {\tiny $0$ }{  $(4)$} & {\tiny $0$ }{  $(4)$} & {\tiny $0$ }{  $(4)$} & {\tiny $0$ }{  $(4)$} & {\tiny $0$ }{  $(2)$} & {\tiny $0$ } &  &  &  &  \\ \hline
$5$ &{\tiny $0$ } & {\tiny $0$ } & {\tiny $0$ } & {\tiny $0$ } & {\tiny $0$ } & {\tiny $0$ } & {\tiny $0$ } &  &  &  &  &  \\ \hline
$6$ &{\tiny $0$ } & {\tiny $0$ } & {\tiny $0$ } & {\tiny $0$ } & {\tiny $0$ } & {\tiny $0$ } &  &  &  &  &  &  \\ \hline
$7$ &{\tiny $0$ } & {\tiny $0$ } & {\tiny $0$ } & {\tiny $0$ } & {\tiny $0$ } &  &  &  &  &  &  &  \\ \hline
$8$ &{\tiny $0$ } & {\tiny $0$ } & {\tiny $0$ } & {\tiny $0$ } &  &  &  &  &  &  &  &  \\ \hline
$9$ &{\tiny $0$ } & {\tiny $0$ } & {\tiny $0$ } &  &  &  &  &  &  &  &  &  \\ \hline
$10$ &{\tiny $0$ } & {\tiny $0$ } &  &  &  &  &  &  &  &  &  &  \\ \hline
$11$ &{\tiny $0$ } &  &  &  &  &  &  &  &  &  &  &  \\ \hline
\end{tabular}
\end{table}

\begin{table} [H]

\footnotesize
\caption{$N_3:\ \mathbb{Z} \langle x,y \rangle /(x^3,y^4)$, Time: 912.87 sec, Memory: 789.53MB}
\begin{tabular} {|l||l|l|l|l|l|l|l|l|l|l|l|l|}
\hline
$(m,n)$ & $0$ & $1$ & $2$ & $3$ & $4$ & $5$ & $6$ & $7$ & $8$ & $9$ & $10$ & $11$ \\ \hline \hline
$0$ &{\tiny $0$ } & {\tiny $0$ } & {\tiny $0$ } & {\tiny $0$ } & {\tiny $0$ } & {\tiny $0$ } & {\tiny $0$ } & {\tiny $0$ } & {\tiny $0$ } & {\tiny $0$ } & {\tiny $0$ } & {\tiny $0$ } \\ \hline
$1$ &{\tiny $0$ } & {\tiny $0$ } & $1$ & $1$ & $1$ & {\tiny $0$ }{  $(2)$} & {\tiny $0$ } & {\tiny $0$ } & {\tiny $0$ } & {\tiny $0$ } & {\tiny $0$ } &  \\ \hline
$2$ &{\tiny $0$ } & $1$ & $3$ & $3$ & $2${  $(4)$} & {\tiny $0$ }{  $(2 \cdot 4)$} & {\tiny $0$ } & {\tiny $0$ } & {\tiny $0$ } & {\tiny $0$ } &  &  \\ \hline
$3$ &{\tiny $0$ } & $1$ & $2${  $(3)$} & $2${  $(3)$} & $1${  $(3 \cdot 4)$} & {\tiny $0$ }{  $(4)$} & {\tiny $0$ } & {\tiny $0$ } & {\tiny $0$ } &  &  &  \\ \hline
$4$ &{\tiny $0$ } & {\tiny $0$ }{  $(3)$} & {\tiny $0$ }{  $(3^{2})$} & {\tiny $0$ }{  $(3^{2})$} & {\tiny $0$ }{  $(3)$} & {\tiny $0$ } & {\tiny $0$ } & {\tiny $0$ } &  &  &  &  \\ \hline
$5$ &{\tiny $0$ } & {\tiny $0$ } & {\tiny $0$ } & {\tiny $0$ } & {\tiny $0$ } & {\tiny $0$ } & {\tiny $0$ } &  &  &  &  &  \\ \hline
$6$ &{\tiny $0$ } & {\tiny $0$ } & {\tiny $0$ } & {\tiny $0$ } & {\tiny $0$ } & {\tiny $0$ } &  &  &  &  &  &  \\ \hline
$7$ &{\tiny $0$ } & {\tiny $0$ } & {\tiny $0$ } & {\tiny $0$ } & {\tiny $0$ } &  &  &  &  &  &  &  \\ \hline
$8$ &{\tiny $0$ } & {\tiny $0$ } & {\tiny $0$ } & {\tiny $0$ } &  &  &  &  &  &  &  &  \\ \hline
$9$ &{\tiny $0$ } & {\tiny $0$ } & {\tiny $0$ } &  &  &  &  &  &  &  &  &  \\ \hline
$10$ &{\tiny $0$ } & {\tiny $0$ } &  &  &  &  &  &  &  &  &  &  \\ \hline
$11$ &{\tiny $0$ } &  &  &  &  &  &  &  &  &  &  &  \\ \hline
\end{tabular}
\end{table}

\begin{table} [H]
\label{ref2ExmpTable}
\footnotesize
\caption{$N_3:\ \mathbb{Z}_3 \langle x,y \rangle /(x^3,y^4)$, Time: 97654.05 sec, Memory: 2783.16MB}
\begin{tabular} {|l||l|l|l|l|l|l|l|l|l|l|l|l|}
\hline
$(m,n)$ & $0$ & $1$ & $2$ & $3$ & $4$ & $5$ & $6$ & $7$ & $8$ & $9$ & $10$ & $11$ \\ \hline \hline
$0$ &{\tiny $0$ } & {\tiny $0$ } & {\tiny $0$ } & {\tiny $0$ } & {\tiny $0$ } & {\tiny $0$ } & {\tiny $0$ } & {\tiny $0$ } & {\tiny $0$ } & {\tiny $0$ } & {\tiny $0$ } & {\tiny $0$ } \\ \hline
$1$ &{\tiny $0$ } & {\tiny $0$ } & $1$ & $1$ & $1$ & {\tiny $0$ } & {\tiny $0$ } & {\tiny $0$ } & {\tiny $0$ } & {\tiny $0$ } & {\tiny $0$ } &  \\ \hline
$2$ &{\tiny $0$ } & $1$ & $3$ & $3$ & $2$ & {\tiny $0$ } & {\tiny $0$ } & {\tiny $0$ } & {\tiny $0$ } & {\tiny $0$ } &  &  \\ \hline
$3$ &{\tiny $0$ } & $1$ & $3$ & $3$ & $2$ & {\tiny $0$ } & {\tiny $0$ } & {\tiny $0$ } & {\tiny $0$ } &  &  &  \\ \hline
$4$ &{\tiny $0$ } & $1$ & $2$ & $2$ & $1$ & {\tiny $0$ } & {\tiny $0$ } & {\tiny $0$ } &  &  &  &  \\ \hline
$5$ &{\tiny $0$ } & {\tiny $0$ } & {\tiny $0$ } & {\tiny $0$ } & {\tiny $0$ } & {\tiny $0$ } & {\tiny $0$ } &  &  &  &  &  \\ \hline
$6$ &{\tiny $0$ } & {\tiny $0$ } & {\tiny $0$ } & {\tiny $0$ } & {\tiny $0$ } & {\tiny $0$ } &  &  &  &  &  &  \\ \hline
$7$ &{\tiny $0$ } & {\tiny $0$ } & {\tiny $0$ } & {\tiny $0$ } & {\tiny $0$ } &  &  &  &  &  &  &  \\ \hline
$8$ &{\tiny $0$ } & {\tiny $0$ } & {\tiny $0$ } & {\tiny $0$ } &  &  &  &  &  &  &  &  \\ \hline
$9$ &{\tiny $0$ } & {\tiny $0$ } & {\tiny $0$ } &  &  &  &  &  &  &  &  &  \\ \hline
$10$ &{\tiny $0$ } & {\tiny $0$ } &  &  &  &  &  &  &  &  &  &  \\ \hline
$11$ &{\tiny $0$ } &  &  &  &  &  &  &  &  &  &  &  \\ \hline
\end{tabular}
\end{table}

\begin{table} [H]
\caption{$N_3:\ \mathbb{Z} \langle x,y \rangle /(x^7,y^7)$, Time: 879.42 sec, Memory: 754.81MB}
\begin{tabular} {|l||l|l|l|l|l|l|l|l|l|l|l|l|}
\hline
$(m,n)$ & $0$ & $1$ & $2$ & $3$ & $4$ & $5$ & $6$ & $7$ & $8$ & $9$ & $10$ & $11$ \\ \hline \hline
$0$ &{\tiny $0$ } & {\tiny $0$ } & {\tiny $0$ } & {\tiny $0$ } & {\tiny $0$ } & {\tiny $0$ } & {\tiny $0$ } & {\tiny $0$ } & {\tiny $0$ } & {\tiny $0$ } & {\tiny $0$ } & {\tiny $0$ } \\ \hline
$1$ &{\tiny $0$ } & {\tiny $0$ } & $1$ & $1$ & $1$ & $1$ & $1$ & $1$ & {\tiny $0$ }{  $(7)$} & {\tiny $0$ } & {\tiny $0$ } &  \\ \hline
$2$ &{\tiny $0$ } & $1$ & $3$ & $3$ & $3$ & $3$ & $3$ & $2${  $(7)$} & {\tiny $0$ }{  $(7^{2})$} & {\tiny $0$ } &  &  \\ \hline
$3$ &{\tiny $0$ } & $1$ & $3$ & $3$ & $3$ & $3$ & $3$ & $2${  $(7)$} & {\tiny $0$ }{  $(7^{2})$} &  &  &  \\ \hline
$4$ &{\tiny $0$ } & $1$ & $3$ & $3$ & $3$ & $3$ & $3$ & $2${  $(7)$} &  &  &  &  \\ \hline
$5$ &{\tiny $0$ } & $1$ & $3$ & $3$ & $3$ & $3$ & $3$ &  &  &  &  &  \\ \hline
$6$ &{\tiny $0$ } & $1$ & $3$ & $3$ & $3$ & $3$ &  &  &  &  &  &  \\ \hline
$7$ &{\tiny $0$ } & $1$ & $2${  $(7)$} & $2${  $(7)$} & $2${  $(7)$} &  &  &  &  &  &  &  \\ \hline
$8$ &{\tiny $0$ } & {\tiny $0$ }{  $(7)$} & {\tiny $0$ }{  $(7^{2})$} & {\tiny $0$ }{  $(7^{2})$} &  &  &  &  &  &  &  &  \\ \hline
$9$ &{\tiny $0$ } & {\tiny $0$ } & {\tiny $0$ } &  &  &  &  &  &  &  &  &  \\ \hline
$10$ &{\tiny $0$ } & {\tiny $0$ } &  &  &  &  &  &  &  &  &  &  \\ \hline
$11$ &{\tiny $0$ } &  &  &  &  &  &  &  &  &  &  &  \\ \hline
\end{tabular}
\end{table}

\begin{table} [H]
\caption{$N_3:\ \mathbb{Z}_7 \langle x,y \rangle /(x^7,y^7)$, Time: 15927.51 sec, Memory: 4333.34MB}
\begin{tabular} {|l||l|l|l|l|l|l|l|l|l|l|l|l|}
\hline
$(m,n)$ & $0$ & $1$ & $2$ & $3$ & $4$ & $5$ & $6$ & $7$ & $8$ & $9$ & $10$ & $11$ \\ \hline \hline
$(0,n)$ &{\tiny $0$} & {\tiny $0$} & {\tiny $0$} & {\tiny $0$} & {\tiny $0$} & {\tiny $0$} & {\tiny $0$} & {\tiny $0$} & {\tiny $0$} & {\tiny $0$} & {\tiny $0$} & {\tiny $0$} \\ \hline
$(1,n)$ &{\tiny $0$} & {\tiny $0$} & $1$ & $1$ & $1$ & $1$ & $1$ & $1$ & $1$ & {\tiny $0$} & {\tiny $0$} & $$ \\ \hline
$(2,n)$ &{\tiny $0$} & $1$ & $3$ & $3$ & $3$ & $3$ & $3$ & $3$ & $2$ & {\tiny $0$} & $$ & $$ \\ \hline
$(3,n)$ &{\tiny $0$} & $1$ & $3$ & $3$ & $3$ & $3$ & $3$ & $3$ & $2$ & $$ & $$ & $$ \\ \hline
$(4,n)$ &{\tiny $0$} & $1$ & $3$ & $3$ & $3$ & $3$ & $3$ & $3$ & $$ & $$ & $$ & $$ \\ \hline
$(5,n)$ &{\tiny $0$} & $1$ & $3$ & $3$ & $3$ & $3$ & $3$ & $$ & $$ & $$ & $$ & $$ \\ \hline
$(6,n)$ &{\tiny $0$} & $1$ & $3$ & $3$ & $3$ & $3$ & $$ & $$ & $$ & $$ & $$ & $$ \\ \hline
$(7,n)$ &{\tiny $0$} & $1$ & $3$ & $3$ & $3$ & $$ & $$ & $$ & $$ & $$ & $$ & $$ \\ \hline
$(8,n)$ &{\tiny $0$} & $1$ & $2$ & $2$ & $$ & $$ & $$ & $$ & $$ & $$ & $$ & $$ \\ \hline
$(9,n)$ &{\tiny $0$} & {\tiny $0$} & {\tiny $0$} & $$ & $$ & $$ & $$ & $$ & $$ & $$ & $$ & $$ \\ \hline
$(10,n)$ &{\tiny $0$} & {\tiny $0$} & $$ & $$ & $$ & $$ & $$ & $$ & $$ & $$ & $$ & $$ \\ \hline
$(11,n)$ &{\tiny $0$} & $$ & $$ & $$ & $$ & $$ & $$ & $$ & $$ & $$ & $$ & $$ \\ \hline
\end{tabular}
\end{table}

\begin{table} [H]
\caption{$N_3:\ \mathbb{Z} \langle x,y \rangle /(x^8,y^8)$, Time: 876.37 sec, Memory: 754.19MB}
\begin{tabular} {|l||l|l|l|l|l|l|l|l|l|l|l|l|}
\hline
$(m,n)$ & $0$ & $1$ & $2$ & $3$ & $4$ & $5$ & $6$ & $7$ & $8$ & $9$ & $10$ & $11$ \\ \hline \hline
$0$ &{\tiny $0$ } & {\tiny $0$ } & {\tiny $0$ } & {\tiny $0$ } & {\tiny $0$ } & {\tiny $0$ } & {\tiny $0$ } & {\tiny $0$ } & {\tiny $0$ } & {\tiny $0$ } & {\tiny $0$ } & {\tiny $0$ } \\ \hline
$1$ &{\tiny $0$ } & {\tiny $0$ } & $1$ & $1$ & $1$ & $1$ & $1$ & $1$ & $1$ & {\tiny $0$ }{  $(4)$} & {\tiny $0$ } &  \\ \hline
$2$ &{\tiny $0$ } & $1$ & $3$ & $3$ & $3$ & $3$ & $3$ & $3$ & $2${  $(8)$} & {\tiny $0$ }{  $(4 \cdot 8)$} &  &  \\ \hline
$3$ &{\tiny $0$ } & $1$ & $3$ & $3$ & $3$ & $3$ & $3$ & $3$ & $2${  $(8)$} &  &  &  \\ \hline
$4$ &{\tiny $0$ } & $1$ & $3$ & $3$ & $3$ & $3$ & $3$ & $3$ &  &  &  &  \\ \hline
$5$ &{\tiny $0$ } & $1$ & $3$ & $3$ & $3$ & $3$ & $3$ &  &  &  &  &  \\ \hline
$6$ &{\tiny $0$ } & $1$ & $3$ & $3$ & $3$ & $3$ &  &  &  &  &  &  \\ \hline
$7$ &{\tiny $0$ } & $1$ & $3$ & $3$ & $3$ &  &  &  &  &  &  &  \\ \hline
$8$ &{\tiny $0$ } & $1$ & $2${  $(8)$} & $2${  $(8)$} &  &  &  &  &  &  &  &  \\ \hline
$9$ &{\tiny $0$ } & {\tiny $0$ }{  $(4)$} & {\tiny $0$ }{  $(4 \cdot 8)$} &  &  &  &  &  &  &  &  &  \\ \hline
$10$ &{\tiny $0$ } & {\tiny $0$ } &  &  &  &  &  &  &  &  &  &  \\ \hline
$11$ &{\tiny $0$ } &  &  &  &  &  &  &  &  &  &  &  \\ \hline
\end{tabular}
\end{table}

\begin{table} [H]
\caption{$N_3:\ \mathbb{Z} \langle x,y \rangle /(x^8,y^9)$, Time: 877.02 sec, Memory: 753.88MB}
\begin{tabular} {|l||l|l|l|l|l|l|l|l|l|l|l|l|}
\hline
$(m,n)$ & $0$ & $1$ & $2$ & $3$ & $4$ & $5$ & $6$ & $7$ & $8$ & $9$ & $10$ & $11$ \\ \hline \hline
$0$ &{\tiny $0$ } & {\tiny $0$ } & {\tiny $0$ } & {\tiny $0$ } & {\tiny $0$ } & {\tiny $0$ } & {\tiny $0$ } & {\tiny $0$ } & {\tiny $0$ } & {\tiny $0$ } & {\tiny $0$ } & {\tiny $0$ } \\ \hline
$1$ &{\tiny $0$ } & {\tiny $0$ } & $1$ & $1$ & $1$ & $1$ & $1$ & $1$ & $1$ & $1$ & {\tiny $0$ }{  $(9)$} &  \\ \hline
$2$ &{\tiny $0$ } & $1$ & $3$ & $3$ & $3$ & $3$ & $3$ & $3$ & $3$ & $2${  $(9)$} &  &  \\ \hline
$3$ &{\tiny $0$ } & $1$ & $3$ & $3$ & $3$ & $3$ & $3$ & $3$ & $3$ &  &  &  \\ \hline
$4$ &{\tiny $0$ } & $1$ & $3$ & $3$ & $3$ & $3$ & $3$ & $3$ &  &  &  &  \\ \hline
$5$ &{\tiny $0$ } & $1$ & $3$ & $3$ & $3$ & $3$ & $3$ &  &  &  &  &  \\ \hline
$6$ &{\tiny $0$ } & $1$ & $3$ & $3$ & $3$ & $3$ &  &  &  &  &  &  \\ \hline
$7$ &{\tiny $0$ } & $1$ & $3$ & $3$ & $3$ &  &  &  &  &  &  &  \\ \hline
$8$ &{\tiny $0$ } & $1$ & $2${  $(8)$} & $2${  $(8)$} &  &  &  &  &  &  &  &  \\ \hline
$9$ &{\tiny $0$ } & {\tiny $0$ }{  $(4)$} & {\tiny $0$ }{  $(4 \cdot 8)$} &  &  &  &  &  &  &  &  &  \\ \hline
$10$ &{\tiny $0$ } & {\tiny $0$ } &  &  &  &  &  &  &  &  &  &  \\ \hline
$11$ &{\tiny $0$ } &  &  &  &  &  &  &  &  &  &  &  \\ \hline
\end{tabular}
\end{table}

\begin{table} [H]
\caption{$N_4:\ \mathbb{Z} \langle x,y \rangle /(x^5,y^4)$, Time: 524.7 sec, Memory: 772.22MB}
\begin{tabular} {|l||l|l|l|l|l|l|l|l|l|l|l|l|}
\hline
$(m,n)$ & $0$ & $1$ & $2$ & $3$ & $4$ & $5$ & $6$ & $7$ & $8$ & $9$ & $10$ & $11$ \\ \hline \hline
$0$ &{\tiny $0$ } & {\tiny $0$ } & {\tiny $0$ } & {\tiny $0$ } & {\tiny $0$ } & {\tiny $0$ } & {\tiny $0$ } & {\tiny $0$ } & {\tiny $0$ } & {\tiny $0$ } & {\tiny $0$ } & {\tiny $0$ } \\ \hline
$1$ &{\tiny $0$ } & {\tiny $0$ } & {\tiny $0$ } & $1$ & $1$ & {\tiny $0$ }{\scriptsize $(2 \cdot 5)$} & {\tiny $0$ }{\scriptsize $(2)$} & {\tiny $0$ } & {\tiny $0$ } & {\tiny $0$ } & {\tiny $0$ } &  \\ \hline
$2$ &{\tiny $0$ } & {\tiny $0$ } & $1$ & $3$ & $3$ & {\tiny $0$ }{\scriptsize $(2^{2} \cdot 3 \cdot 4 \cdot 5)$} & {\tiny $0$ }{\scriptsize $(2^{2})$} & {\tiny $0$ } & {\tiny $0$ } & {\tiny $0$ } &  &  \\ \hline
$3$ &{\tiny $0$ } & $1$ & $3$ & $6$ & $5${\scriptsize $(4)$} & {\tiny $0$ }{\scriptsize $(2^{2} \cdot 4^{3} \cdot 3^{2})$} & {\tiny $0$ }{\scriptsize $(2^{2} \cdot 4)$} & {\tiny $0$ } & {\tiny $0$ } &  &  &  \\ \hline
$4$ &{\tiny $0$ } & $1$ & $3$ & $6$ & $5${\scriptsize $(4)$} & {\tiny $0$ }{\scriptsize $(2^{2} \cdot 4^{3} \cdot 3^{2})$} & {\tiny $0$ }{\scriptsize $(2^{2} \cdot 4)$} & {\tiny $0$ } &  &  &  &  \\ \hline
$5$ &{\tiny $0$ } & $1$ & $3$ & $5${\scriptsize $(5)$} & $4${\scriptsize $(4 \cdot 5)$} & {\tiny $0$ }{\scriptsize $(2 \cdot 4^{3} \cdot 3^{2})$} & {\tiny $0$ }{\scriptsize $(2 \cdot 4)$} &  &  &  &  &  \\ \hline
$6$ &{\tiny $0$ } & {\tiny $0$ }{\scriptsize $(2 \cdot 5)$} & {\tiny $0$ }{\scriptsize $(5^{3} \cdot 2 \cdot 4)$} & {\tiny $0$ }{\scriptsize $(5^{5} \cdot 2 \cdot 4^{2})$} & {\tiny $0$ }{\scriptsize $(5^{4} \cdot 2 \cdot 4^{2})$} & {\tiny $0$ }{\scriptsize $(4^{2})$} &  &  &  &  &  &  \\ \hline
$7$ &{\tiny $0$ } & {\tiny $0$ }{\scriptsize $(5)$} & {\tiny $0$ }{\scriptsize $(5^{2})$} & {\tiny $0$ }{\scriptsize $(5^{3})$} & {\tiny $0$ }{\scriptsize $(5^{2})$} &  &  &  &  &  &  &  \\ \hline
$8$ &{\tiny $0$ } & {\tiny $0$ } & {\tiny $0$ } & {\tiny $0$ } &  &  &  &  &  &  &  &  \\ \hline
$9$ &{\tiny $0$ } & {\tiny $0$ } & {\tiny $0$ } &  &  &  &  &  &  &  &  &  \\ \hline
$10$ &{\tiny $0$ } & {\tiny $0$ } &  &  &  &  &  &  &  &  &  &  \\ \hline
$11$ &{\tiny $0$ } &  &  &  &  &  &  &  &  &  &  &  \\ \hline
\end{tabular}
\end{table}

\begin{table} [H]
\caption{$N_4:\ \mathbb{Z} \langle x,y \rangle /(x^{101})$, Time: 878.2 sec, Memory: 753.88MB}
\begin{tabular} {|l||l|l|l|l|l|l|l|l|l|l|l|l|}
\hline
$(m,n)$ & $0$ & $1$ & $2$ & $3$ & $4$ & $5$ & $6$ & $7$ & $8$ & $9$ & $10$ & $11$ \\ \hline \hline
$0$ &{\tiny $0$ } & {\tiny $0$ } & {\tiny $0$ } & {\tiny $0$ } & {\tiny $0$ } & {\tiny $0$ } & {\tiny $0$ } & {\tiny $0$ } & {\tiny $0$ } & {\tiny $0$ } & {\tiny $0$ } & {\tiny $0$ } \\ \hline
$1$ &{\tiny $0$ } & {\tiny $0$ } & {\tiny $0$ } & $1$ & $1$ & $1$ & $1$ & $1$ & $1$ & $1$ & $1$ &  \\ \hline
$2$ &{\tiny $0$ } & {\tiny $0$ } & $1$ & $3$ & $3$ & $3$ & $3$ & $3$ & $3$ & $3$ &  &  \\ \hline
$3$ &{\tiny $0$ } & $1$ & $3$ & $6$ & $6$ & $6$ & $6$ & $6$ & $6$ &  &  &  \\ \hline
$4$ &{\tiny $0$ } & $1$ & $3$ & $6$ & $6$ & $6$ & $6$ & $6$ &  &  &  &  \\ \hline
$5$ &{\tiny $0$ } & $1$ & $3$ & $6$ & $6$ & $6$ & $6$ &  &  &  &  &  \\ \hline
$6$ &{\tiny $0$ } & $1$ & $3$ & $6$ & $6$ & $6$ &  &  &  &  &  &  \\ \hline
$7$ &{\tiny $0$ } & $1$ & $3$ & $6$ & $6$ &  &  &  &  &  &  &  \\ \hline
$8$ &{\tiny $0$ } & $1$ & $3$ & $6$ &  &  &  &  &  &  &  &  \\ \hline
$9$ &{\tiny $0$ } & $1$ & $3$ &  &  &  &  &  &  &  &  &  \\ \hline
$10$ &{\tiny $0$ } & $1$ &  &  &  &  &  &  &  &  &  &  \\ \hline
$11$ &{\tiny $0$ } &  &  &  &  &  &  &  &  &  &  &  \\ \hline
\end{tabular}
\end{table}

\begin{table} [H]
\caption{$N_4:\ \mathbb{Z} \langle x,y \rangle /(x^3,y^3)$, Time: 1730.05 sec, Memory: 1582.34MB}
\begin{tabular} {|l||l|l|l|l|l|l|l|l|l|l|l|l|}
\hline
$(m,n)$ & $0$ & $1$ & $2$ & $3$ & $4$ & $5$ & $6$ & $7$ & $8$ & $9$ & $10$ & $11$ \\ \hline \hline
$0$ &{\tiny $0$ } & {\tiny $0$ } & {\tiny $0$ } & {\tiny $0$ } & {\tiny $0$ } & {\tiny $0$ } & {\tiny $0$ } & {\tiny $0$ } & {\tiny $0$ } & {\tiny $0$ } & {\tiny $0$ } & {\tiny $0$ } \\ \hline
$1$ &{\tiny $0$ } & {\tiny $0$ } & {\tiny $0$ } & $1$ & {\tiny $0$ }{  $(2)$} & {\tiny $0$ } & {\tiny $0$ } & {\tiny $0$ } & {\tiny $0$ } & {\tiny $0$ } & {\tiny $0$ } &  \\ \hline
$2$ &{\tiny $0$ } & {\tiny $0$ } & $1$ & $3$ & {\tiny $0$ }{  $(2^{2} \cdot 3^{2})$} & {\tiny $0$ }{  $(3)$} & {\tiny $0$ } & {\tiny $0$ } & {\tiny $0$ } & {\tiny $0$ } &  &  \\ \hline
$3$ &{\tiny $0$ } & $1$ & $3$ & $4${  $(3)$} & {\tiny $0$ }{  $(3^{4} \cdot 2^{2})$} & {\tiny $0$ }{  $(3^{2})$} & {\tiny $0$ } & {\tiny $0$ } & {\tiny $0$ } &  &  &  \\ \hline
$4$ &{\tiny $0$ } & {\tiny $0$ }{  $(2)$} & {\tiny $0$ }{  $(2^{2} \cdot 3^{2})$} & {\tiny $0$ }{  $(3^{4} \cdot 2^{2})$} & {\tiny $0$ }{  $(3^{4} \cdot 2)$} & {\tiny $0$ }{  $(3^{2})$} & {\tiny $0$ } & {\tiny $0$ } &  &  &  &  \\ \hline
$5$ &{\tiny $0$ } & {\tiny $0$ } & {\tiny $0$ }{  $(3)$} & {\tiny $0$ }{  $(3^{2})$} & {\tiny $0$ }{  $(3^{2})$} & {\tiny $0$ }{  $(3)$} & {\tiny $0$ } &  &  &  &  &  \\ \hline
$6$ &{\tiny $0$ } & {\tiny $0$ } & {\tiny $0$ } & {\tiny $0$ } & {\tiny $0$ } & {\tiny $0$ } &  &  &  &  &  &  \\ \hline
$7$ &{\tiny $0$ } & {\tiny $0$ } & {\tiny $0$ } & {\tiny $0$ } & {\tiny $0$ } &  &  &  &  &  &  &  \\ \hline
$8$ &{\tiny $0$ } & {\tiny $0$ } & {\tiny $0$ } & {\tiny $0$ } &  &  &  &  &  &  &  &  \\ \hline
$9$ &{\tiny $0$ } & {\tiny $0$ } & {\tiny $0$ } &  &  &  &  &  &  &  &  &  \\ \hline
$10$ &{\tiny $0$ } & {\tiny $0$ } &  &  &  &  &  &  &  &  &  &  \\ \hline
$11$ &{\tiny $0$ } &  &  &  &  &  &  &  &  &  &  &  \\ \hline
\end{tabular}
\end{table}

\begin{table} [H]
\caption{$N_4:\ \mathbb{Z} \langle x,y \rangle /(x^3,y^4)$, Time: 912.87 sec, Memory: 789.53MB}
\begin{tabular} {|l||l|l|l|l|l|l|l|l|l|l|l|l|}
\hline
$(m,n)$ & $0$ & $1$ & $2$ & $3$ & $4$ & $5$ & $6$ & $7$ & $8$ & $9$ & $10$ & $11$ \\ \hline \hline
$0$ &{\tiny $0$ } & {\tiny $0$ } & {\tiny $0$ } & {\tiny $0$ } & {\tiny $0$ } & {\tiny $0$ } & {\tiny $0$ } & {\tiny $0$ } & {\tiny $0$ } & {\tiny $0$ } & {\tiny $0$ } & {\tiny $0$ } \\ \hline
$1$ &{\tiny $0$ } & {\tiny $0$ } & {\tiny $0$ } & $1$ & $1$ & {\tiny $0$ }{  $(2 \cdot 5)$} & {\tiny $0$ }{  $(2)$} & {\tiny $0$ } & {\tiny $0$ } & {\tiny $0$ } & {\tiny $0$ } &  \\ \hline
$2$ &{\tiny $0$ } & {\tiny $0$ } & $1$ & $3$ & $3$ & {\tiny $0$ }{  $(2^{2} \cdot 3 \cdot 4 \cdot 5)$} & {\tiny $0$ }{  $(2^{2})$} & {\tiny $0$ } & {\tiny $0$ } & {\tiny $0$ } &  &  \\ \hline
$3$ &{\tiny $0$ } & $1$ & $3$ & $5${  $(3)$} & $4${  $(3 \cdot 4)$} & {\tiny $0$ }{  $(2 \cdot 4^{3} \cdot 3^{2})$} & {\tiny $0$ }{  $(2 \cdot 4)$} & {\tiny $0$ } & {\tiny $0$ } &  &  &  \\ \hline
$4$ &{\tiny $0$ } & {\tiny $0$ }{  $(2)$} & {\tiny $0$ }{  $(2^{2} \cdot 3^{2})$} & {\tiny $0$ }{  $(3^{4} \cdot 2^{3})$} & {\tiny $0$ }{  $(3^{4} \cdot 2^{3})$} & {\tiny $0$ }{  $(2^{2} \cdot 3^{2})$} & {\tiny $0$ }{  $(2)$} & {\tiny $0$ } &  &  &  &  \\ \hline
$5$ &{\tiny $0$ } & {\tiny $0$ } & {\tiny $0$ }{  $(3)$} & {\tiny $0$ }{  $(3^{2})$} & {\tiny $0$ }{  $(3^{2})$} & {\tiny $0$ }{  $(3)$} & {\tiny $0$ } &  &  &  &  &  \\ \hline
$6$ &{\tiny $0$ } & {\tiny $0$ } & {\tiny $0$ } & {\tiny $0$ } & {\tiny $0$ } & {\tiny $0$ } &  &  &  &  &  &  \\ \hline
$7$ &{\tiny $0$ } & {\tiny $0$ } & {\tiny $0$ } & {\tiny $0$ } & {\tiny $0$ } &  &  &  &  &  &  &  \\ \hline
$8$ &{\tiny $0$ } & {\tiny $0$ } & {\tiny $0$ } & {\tiny $0$ } &  &  &  &  &  &  &  &  \\ \hline
$9$ &{\tiny $0$ } & {\tiny $0$ } & {\tiny $0$ } &  &  &  &  &  &  &  &  &  \\ \hline
$10$ &{\tiny $0$ } & {\tiny $0$ } &  &  &  &  &  &  &  &  &  &  \\ \hline
$11$ &{\tiny $0$ } &  &  &  &  &  &  &  &  &  &  &  \\ \hline
\end{tabular}
\end{table}

\begin{table} [H]
\footnotesize
\caption{$N_5:\ \mathbb{Z} \langle x,y \rangle /(x^3,y^4)$, Time: 912.87 sec, Memory: 789.53MB}
\begin{tabular} {|l||l|l|l|l|l|l|l|l|l|l|l|l|}
\hline
$(m,n)$ & $0$ & $1$ & $2$ & $3$ & $4$ & $5$ & $6$ & $7$ & $8$ & $9$ & $10$ & $11$ \\ \hline \hline
$0$ &{\tiny $0$ } & {\tiny $0$ } & {\tiny $0$ } & {\tiny $0$ } & {\tiny $0$ } & {\tiny $0$ } & {\tiny $0$ } & {\tiny $0$ } & {\tiny $0$ } & {\tiny $0$ } & {\tiny $0$ } & {\tiny $0$ } \\ \hline
$1$ &{\tiny $0$ } & {\tiny $0$ } & {\tiny $0$ } & {\tiny $0$ } & $1$ & $1$ & {\tiny $0$ }{  $(5)$} & {\tiny $0$ } & {\tiny $0$ } & {\tiny $0$ } & {\tiny $0$ } &  \\ \hline
$2$ &{\tiny $0$ } & {\tiny $0$ } & {\tiny $0$ } & $2$ & $5$ & $4${  $(2 \cdot 5)$} & {\tiny $0$ }{  $(2^{4})$} & {\tiny $0$ }{  $(2^{2})$} & {\tiny $0$ } & {\tiny $0$ } &  &  \\ \hline
$3$ &{\tiny $0$ } & {\tiny $0$ } & $2$ & $6$ & $9${  $(3)$} & $5${  $(2^{2} \cdot 3^{2} \cdot 4^{2})$} & {\tiny $0$ }{  $(2^{6} \cdot 3 \cdot 4 \cdot 5)$} & {\tiny $0$ }{  $(2^{3})$} & {\tiny $0$ } &  &  &  \\ \hline
$4$ &{\tiny $0$ } & $1$ & $4${  $(2)$} & $6${  $(2^{2} \cdot 3^{3})$} & $6${  $(2^{4} \cdot 3^{6})$} & $2${  $(2^{5} \cdot 3^{2} \cdot 4^{2})$} & {\tiny $0$ }{  $(2^{4} \cdot 3 \cdot 4)$} & {\tiny $0$ }{  $(2^{2})$} &  &  &  &  \\ \hline
$5$ &{\tiny $0$ } & {\tiny $0$ } & {\tiny $0$ }{  $(3^{2} \cdot 2)$} & {\tiny $0$ }{  $(3^{5} \cdot 2^{2})$} & {\tiny $0$ }{  $(3^{7} \cdot 2^{3})$} & {\tiny $0$ }{  $(3^{5} \cdot 2^{3})$} & {\tiny $0$ }{  $(2^{2} \cdot 3^{2})$} &  &  &  &  &  \\ \hline
$6$ &{\tiny $0$ } & {\tiny $0$ } & {\tiny $0$ }{  $(3)$} & {\tiny $0$ }{  $(3^{2})$} & {\tiny $0$ }{  $(3^{3})$} & {\tiny $0$ }{  $(3^{2})$} &  &  &  &  &  &  \\ \hline
$7$ &{\tiny $0$ } & {\tiny $0$ } & {\tiny $0$ } & {\tiny $0$ } & {\tiny $0$ } &  &  &  &  &  &  &  \\ \hline
$8$ &{\tiny $0$ } & {\tiny $0$ } & {\tiny $0$ } & {\tiny $0$ } &  &  &  &  &  &  &  &  \\ \hline
$9$ &{\tiny $0$ } & {\tiny $0$ } & {\tiny $0$ } &  &  &  &  &  &  &  &  &  \\ \hline
$10$ &{\tiny $0$ } & {\tiny $0$ } &  &  &  &  &  &  &  &  &  &  \\ \hline
$11$ &{\tiny $0$ } &  &  &  &  &  &  &  &  &  &  &  \\ \hline
\end{tabular}
\end{table}

\section{Acknowledgements}
We would like to thank P. Etingof for suggesting and advising the project and conjectures, and T. Khovanova for a careful reading of a draft of this paper and suggestions to improve the presentation. Y.F. acknowledges support from Marie Curie Grant IOF-hqsmcf-274032 and from the Max Planck Institute for Mathematics. This research was partly supported by the NCCR SwissMAP, funded by the Swiss National Science Foundation. I. X. thanks the PRIMES program.

\pagebreak
\bibliographystyle{alpha}
\bibliography{paper}

\end{document}